\newif\ifarxiv
\arxivtrue
\ifarxiv
  \documentclass[11pt,letterpaper]{article}
  \usepackage[margin=1in]{geometry}
  \newenvironment{acks}[1][Acknowledgments]{\section*{#1}\pdfbookmark[1]{#1}{sec:ack}}{\par}
  \newenvironment{funding}{\par}{\par} %
\else
  \documentclass[aos]{imsart}
\fi

\RequirePackage[T1]{fontenc}
\RequirePackage[utf8]{inputenc}
\RequirePackage[canadian]{babel}
\RequirePackage{csquotes}

\RequirePackage{amsthm,amsmath,amssymb}
\RequirePackage{mathtools}
\RequirePackage[store-sets-label]{keytheorems}
\RequirePackage{dsfont}
\RequirePackage{enumitem}
\newlist{romanenum}{enumerate}{1}
\setlist[romanenum]{label=(\roman*),ref=(\roman*),leftmargin=*,widest=iii,align=right}   %

\RequirePackage{tikz}
\usetikzlibrary{decorations.pathreplacing}

\ifarxiv\else
  \makeatletter\let\c@author\undefined\makeatother
\fi
\RequirePackage[
  backend=biber,
  style=authoryear-comp,
  maxcitenames=2,
  maxbibnames=99,
  giveninits=true,
  uniquename=false,
  uniquelist=minyear,
  dashed=false,
  doi=true,
  eprint=true,
  isbn=false,
  sortcites=false,
  natbib,
]{biblatex}
\ifarxiv
\DeclareNameAlias{sortname}{given-family}
\renewbibmacro*{in:}{}
\DefineBibliographyExtras{canadian}{\uspunctuation}
\DeclareFieldFormat[misc]{title}{\mkbibquote{#1}}
\else
\renewbibmacro*{in:}{}
\DeclareNameAlias{sortname}{family-given}                %
\AtBeginBibliography{%
  \renewcommand*{\mkbibnamefamily}[1]{\textsc{#1}}%
  \renewcommand*{\mkbibnamegiven}[1]{\textsc{#1}}%
  \renewcommand*{\mkbibnameprefix}[1]{\textsc{#1}}%
  \renewcommand*{\mkbibnamesuffix}[1]{\textsc{#1}}%
}
\DeclareFieldFormat[article,inproceedings,incollection,thesis,unpublished,misc]{title}{#1}  %
\DeclareFieldFormat{journaltitle}{\mkbibemph{#1}}
\DeclareFieldFormat[article]{volume}{\mkbibbold{#1}}
\DeclareFieldFormat[article]{pages}{#1}
\renewbibmacro*{journal+issuetitle}{%
  \usebibmacro{journal}%
  \setunit*{\addspace}%
  \printfield{volume}%
  \setunit*{\addspace}%
  \printfield{pages}%
  \setunit*{\addspace}%
  \printfield{eid}%
  \newunit}
\renewbibmacro*{note+pages}{\printfield{note}\newunit}
\renewbibmacro*{publisher+location+date}{%
  \printlist{publisher}%
  \setunit*{\addcomma\space}%
  \printlist{location}%
  \newunit}
\DeclareFieldFormat{doi}{\url{https://doi.org/#1}}
\defbibheading{bibliography}[\refname]{\specialsection*{#1}}   %
\makeatletter\ifdefined\thebibliography@size\renewcommand*{\bibfont}{\thebibliography@size}\fi\makeatother
\DeclareFieldFormat{eprint:arxiv}{Available at \href{https://arxiv.org/abs/#1}{arXiv\addcolon\nolinkurl{#1}}}
\fi
\RequirePackage[colorlinks=true,allcolors=blue!55!black]{hyperref}
\RequirePackage[capitalize,noabbrev]{cleveref}
\crefformat{enumi}{#2#1#3}\Crefformat{enumi}{#2#1#3}
\crefformat{romanenumi}{#2#1#3}\Crefformat{romanenumi}{#2#1#3}  %

\crefname{appendix}{Appendix}{Appendices}
\Crefname{appendix}{Appendix}{Appendices}

\theoremstyle{plain}
\newkeytheorem{theorem}[numberwithin=section]
\newkeytheorem{corollary}[sibling=theorem]
\newkeytheorem{lemma}[sibling=theorem]
\newkeytheorem{proposition}[sibling=theorem]
\newkeytheorem{fact}[sibling=theorem,refname={Fact,Facts}]
\newkeytheorem{definition}[sibling=theorem,style=definition]

\definecolor{dens0}{HTML}{F4F6F8}   %
\definecolor{dens2}{HTML}{9FC0DA}   %
\definecolor{dens4}{HTML}{1F4E79}   %
\definecolor{rpos}{HTML}{E3D9BE}    %
\definecolor{rneg}{HTML}{6E634D}    %

\newcommand{\cellzero}[2]{\fill[dens0] (#1,#2) rectangle ++(1,1);
  \node[font=\footnotesize,text=black!45] at (#1+.5,#2+.5) {$0$};}
\newcommand{\celltwo}[2]{\fill[dens2] (#1,#2) rectangle ++(1,1);
  \node[font=\footnotesize] at (#1+.5,#2+.5) {$2$};}
\newcommand{\cellfour}[2]{\fill[dens4] (#1,#2) rectangle ++(1,1);
  \node[font=\footnotesize,text=white] at (#1+.5,#2+.5) {$4$};}

\newcommand{\panelframe}[1]{%
  \draw[white,line width=.7pt] (1,0)--(1,2) (0,1)--(2,1);
  \draw[line width=.7pt] (0,0) rectangle (2,2);
  \foreach \t/\l in {0/{$0$},1/{$\tfrac12$},2/{$1$}}{%
    \draw[line width=.5pt] (\t,0)--(\t,-.09);
    \node[font=\scriptsize,below=0pt] at (\t,-.09) {\l};
    \draw[line width=.5pt] (0,\t)--(-.09,\t);
    \node[font=\scriptsize,left=0pt] at (-.09,\t) {\l};}
  \node[font=\scriptsize] at (.5,-.7) {$x^{(1)}$};
  \node[font=\scriptsize,rotate=90] at (-.62,.5) {$y^{(1)}$};
  \node[font=\small,align=center] at (1,2.42) {#1};}

\newcommand{\indep}{\mathrel{\perp\!\!\!\perp}}
\newcommand{\nindep}{\mathrel{\not\!\perp\!\!\!\perp}}
\newcommand{\ud}{\mathop{}\!\mathrm{d}}
\DeclareMathOperator{\E}{\mathbb{E}}
\newcommand{\R}{\mathbb{R}}
\DeclareMathOperator{\indic}{\mathds{1}}
\newcommand{\cA}{\mathcal{A}}
\newcommand{\cB}{\mathcal{B}}
\newcommand{\cD}{\mathcal{D}}
\newcommand{\cE}{\mathcal{E}}
\newcommand{\cM}{\mathcal{M}}
\newcommand{\cP}{\mathcal{P}}
\newcommand{\cQ}{\mathcal{Q}}
\newcommand{\cU}{\mathcal{U}}
\newcommand{\cV}{\mathcal{V}}
\newcommand{\cW}{\mathcal{W}}
\newcommand{\cX}{\mathcal{X}}
\newcommand{\cY}{\mathcal{Y}}
\newcommand{\cZ}{\mathcal{Z}}
\newcommand{\tr}{\mathrm{tr}}
\newcommand{\eval}{\mathrm{e}}
\newcommand{\TV}{\mathrm{TV}}
\DeclareMathOperator{\supp}{supp}
\DeclareMathOperator{\vol}{vol}
\newcommand{\push}{_{\#}}
\newcommand{\tp}{^\mathsf{T}}

\newcommand{\eps}{\varepsilon}

\newcommand{\bone}{\mathbf{1}}

\DeclarePairedDelimiter{\abs}{\lvert}{\rvert}
\DeclarePairedDelimiter{\norm}{\lVert}{\rVert}
\DeclarePairedDelimiter{\floor}{\lfloor}{\rfloor}
\DeclarePairedDelimiter{\ceil}{\lceil}{\rceil}
\DeclarePairedDelimiterX{\inner}[2]{\langle}{\rangle}{#1,#2}

\begin{document}

\ifarxiv\else\begin{frontmatter}\fi
\title{Conditional Independence Is Not (Quite) Pointwise Testable}
\ifarxiv
\author{Danica~J. Sutherland\\
  {\small University of British Columbia}\\
  {\small\href{https://djsutherland.ml}{\texttt{https://djsutherland.ml}}}}
\date{}
\maketitle
\hypersetup{
  pdftitle={Conditional Independence Is Not (Quite) Pointwise Testable},
  pdfauthor={Danica J. Sutherland},
  pdfkeywords={62G10, 62G20, Conditional independence, hypothesis testing, pointwise asymptotics, no-free-lunch theorem}
}
\else
\runtitle{Conditional Independence Is Not (Quite) Pointwise Testable}

\begin{aug}
\author[ubc]{\fnms{Danica~J.}~\snm{Sutherland}\ead[label=e1]{dsuth@cs.ubc.ca}}
\address[ubc]{Department of Computer Science, University of British Columbia\printead[presep={,\ }]{e1}}
\end{aug}
\fi

\begin{abstract}
Shah and Peters showed that a conditional independence test with finite-sample or uniformly-controlled level has only trivial power against any alternative.
Most practical tests, however, only claim pointwise asymptotic level.
There have been incorrect claims in the literature of conditional independence tests with pointwise asymptotic level
and consistency against any alternative;
whether such a test actually exists has remained open.

We resolve this question.
Even restricting to hypotheses with a bounded density on compact subsets of Euclidean spaces,
for any sequence of (possibly randomized) tests
with pointwise asymptotic level $\alpha$,
for every $\varepsilon>0$ there is a conditionally dependent distribution
where the test's limsup power is at most $\alpha+\varepsilon$.
This remains true when the level control is required only over
conditionally independent distributions with a continuous density
and a uniformly continuous conditional law,
and the conditionally dependent distributions have smooth densities.

On the other hand, the extreme limitation
of trivial power against \emph{any} alternative
for uniform-level tests
does not apply to tests with pointwise level.
We exhibit a test that, without regularity assumptions,
has pointwise asymptotic level $\alpha$,
strictly higher power for all alternatives,
and power tending to one for all distributions
whose dependence exceeds a chosen scalar threshold.
\end{abstract}

\ifarxiv
\else


\end{frontmatter}
\fi

\section{Introduction}
\label{sec:introduction}
Consider the problem of testing conditional independence,
\[
  H_0 : X \indep Y \mid Z
,\]
from i.i.d.\ observations $(X_1, Y_1, Z_1), \dots, (X_n, Y_n, Z_n)$.
This problem is of profound importance:
\citet{Dawid1979Sep} identified conditional independence as
a ``language for the expression of statistical concepts and a framework for their study,''
including ``sufficiency and ancillarity, parameter identification, causal inference,
prediction sufficiency, data selection mechanisms, invariant statistical models
and a subjectivist approach to model-building.''
These areas have continued to develop in the ensuing half-century;
measuring or testing conditional dependence
is vital in causal discovery \parencite{pc,gao2026optimalstructurelearningconditional},
graphical modelling \parencite{lauritzen1996},
sufficient dimension reduction \parencite{li1991sir,cook1998},
missing-data mechanisms \parencite{rubin1976},
predictive fairness \parencite{Hardt2016Dec},
variable selection \parencite{Candes2018Jun},
and more.

This testing problem is frequently understood to be quite difficult,
both in practice and in theory.
To understand the difficulties, let us first set some terminology.
For a sample space $\cV$,
let $\psi_n : \cV^n \to [0, 1]$ denote a \emph{randomized test},
a measurable function giving the probability of rejecting the null hypothesis;\footnote{%
  Some sources, such as \textcite{shahPeters2020},
  instead define a test as a function $\phi_n : \cV^n \times [0, 1] \to \{0, 1\}$,
  with the last argument representing a source of randomness $U$.
  The two are equivalent for our purposes;
  we can convert between them as
  $\phi_n(\cdots, u) = \indic(u \le \psi_n(\cdots))$
  and $\psi_n(\cdots) = \int_0^1 \phi_n(\cdots, u) \ud u$.
}
a \emph{deterministic} test takes values only in $\{0, 1\}$.
The probability of rejecting given $n$ i.i.d.\ samples from
a probability measure $\mu$ is $\mu^n \psi_n$,
using the notation $\mu f = \int f(v) \ud \mu(v)$.
For a family of null distributions $\cP$,
the \emph{size} of $\psi_n$ is $\sup_{P \in \cP} P^n \psi_n$,
and $\psi_n$ has \emph{level $\alpha$} if its size is at most $\alpha$.
Let $\psi = (\psi_n)_{n \ge 1}$ be a sequence of tests.
We say
\begin{align*}
  \text{$\psi$ has \emph{finite-sample level $\alpha$}}
  &&\text{ if }&&
  \sup_{n \ge 1} \; \sup_{P \in \cP} \; P^n \psi_n &\le \alpha
  ,\\
  \text{$\psi$ has \emph{uniformly asymptotic level $\alpha$}}
  &&\text{ if }&&
  \limsup_{n \to \infty} \; \sup_{P \in \cP} \; P^n \psi_n &\le \alpha
  ,\\
  \text{$\psi$ has \emph{pointwise asymptotic level $\alpha$}}
  &&\text{ if }&&
  \sup_{P \in \cP} \; \limsup_{n \to \infty} \; P^n \psi_n &\le \alpha
.\end{align*}
Each condition implies those below.
The ideal situation is tests with finite-sample level;
this is sometimes possible, as for e.g.\ many permutation-based tests.
If we have only pointwise asymptotic level,
then for any $n$ there may still exist null distributions
such that, say, $P^n \psi_n > \alpha + 0.5$;
with uniformly asymptotic level,
there is some threshold after which this is no longer possible.
It is worth noting, however,
that if a test sequence $\psi_n$ with uniformly asymptotic level $\alpha>0$ exists,
the test sequence $\psi'_n = \psi_n \, \alpha / \max\left\{ \alpha, \sup_{m \ge n} \sup_{P \in \cP} P^m \psi_m \right\}$ has finite-sample level $\alpha$
and
asymptotically equivalent behaviour:
$\norm{\psi'_n - \psi_n}_\infty \to 0$ uniformly.

Let $\cQ$ be a set of alternatives,
for which we should reject the null hypothesis.
We say
\begin{align*}
  \text{$\psi$ is \emph{pointwise consistent}}
  &&\text{ if }&&
  \inf_{Q \in \cQ} \; \lim_{n \to \infty} \; Q^n \psi_n &= 1
  \\
  \text{$\psi$ is \emph{strongly consistent}}
  &&\text{ if }&&
  \forall \mu \in \cP \cup \cQ,\;\;
  \psi_n &\to \indic(\mu \in \cQ) \text{, $\mu^\infty$-almost surely}
.\end{align*}
Strong consistency implies both pointwise consistency
and pointwise asymptotic level $0$;
for deterministic tests, it is equivalent to stating that a test makes only finitely many mistakes
as it sees an infinite sequence of samples.
One can also define uniform notions of consistency,
but these will not be relevant to our study.

For unconditional independence,
strong results are available \citep{Gretton2010,Rindt2021Dec}.
The following version
is shown in \cref{sec:discrete-z} for completeness.
\getkeytheorem{thm:consistent-ind-test}

What, then, about conditional independence?

\subsection{Finite-sample or uniform level}
Some special cases are benign.

If $Z$ is discrete (concentrated on an at most countable set),
then we can simply split the data by values of $Z$
and run an unconditional independence test on each stratum.
This general idea is well-known \citep[Remark 4]{shahPeters2020},
and Theorem 11(b) of \citet{boekenEtAl2026} gives related guarantees for discrete conditioning spaces.
We prove the following version, allowing arbitrary discrete marginals on a Polish conditioning space, in \cref{sec:discrete-z}.
\getkeytheorem{thm:discrete-consistent}

Another commonly considered special case is the ``Model-X'' regime:
we assume that we know the distribution of $X \mid Z$.
In this case, pointwise-consistent test sequences with finite-sample level
are readily available;
a prominent strategy is to use a conditional randomization test
\parencites[Section 4]{Candes2018Jun}{mimic-and-classify}{berrettEtAl2020}{liuEtAl2022},
whose power is analysed by \textcite{katsevichRamdas2022};
this can be phrased as testing independence
between a data point
and an indicator of whether it was sampled from the true distribution,
or had its $X$ value conditionally resampled given $Z$.
An alternative when both conditionals are known --
which could be relaxed to one conditional with the statistic of \citet{pogodin:circe} --
is given by \citet[Proposition 4.1]{zheng:hardness}.

Alternatively, we can make some explicit assumption on the smoothness of the distribution,
such as that the distribution of $(X, Y) \mid Z$ does not change too quickly in $Z$.
Under various forms of such assumptions,
tests are known that control level uniformly over the assumed class, in some cases even in finite samples,
and are consistent against every fixed alternative in it
\citep{warren2021,neykovEtAl2021,kimEtAl2022,gyorfiEtAl2023,boekenEtAl2026}.
Their power guarantees are usually stated for alternatives separated from the null in total variation.
Any fixed alternative, however, must have some positive total variation distance
from the null class of conditionally independent laws.
\begin{proposition}[{\citealp[Theorem~1]{lauritzen2024}}]
\label{thm:tv-closed}
If $P_m$, $P$ are probability measures on a standard Borel\footnote{%
  \Citeauthor{lauritzen2024} does not actually require a standard Borel space;
that assumption guarantees
that conditional laws exist,
which we will use in
our definition of conditional independence just above \eqref{eq:ci-density}.
} product space $\cX \times \cY \times \cZ$
such that $X \indep_{P_m} Y\mid Z$ for each $m$,
and the total variation distance between $P_m$ and $P$ tends to zero as $m \to \infty$,
then $X \indep_P Y \mid Z$.
\end{proposition}

Under each of these assumptions,
we have positive results.
Some such assumption, however, is necessary.
\Textcite{shahPeters2020}
proved that if we only assume
a Lebesgue joint density on $(X, Y, Z)$, %
the situation is much worse:
a test with uniformly asymptotic level has only trivial power against \emph{any} alternative.
\begin{theorem}[{\citealp[Corollary 3]{shahPeters2020}}] \label{thm:sp}
  Let $d_X, d_Y, d_Z \ge 1$,
  and $\cV = \cX \times \cY \times \cZ = \R^{d_X + d_Y + d_Z}$.
  Picking any $M \in (0, \infty]$,
  let $\cE_M$ be the set of probability measures on $\cV$
  absolutely continuous with respect to Lebesgue measure
  and whose support is in $(-M, M)^{d_X + d_Y + d_Z}$.
  Take $\cP_M \subset \cE_M$ to be the distributions where $X \indep Y \mid Z$.
  If $(\psi_n)$ is a family of tests
  with uniformly asymptotic level $\alpha \in [0, 1)$,
  then $\sup_{Q \in \cE_M \setminus \cP_M} \limsup_{n \to \infty} Q^n \psi_n \le \alpha$;
  thus $(\psi_n)$ is not pointwise consistent.
\end{theorem}
Subsequent hardness results
\parencite[e.g.][]{neykovEtAl2021,kimEtAl2022,lundborgEtAl2022}
are of roughly the same form.

\subsection{Pointwise level}
In most widely-used hypothesis testing settings, however,
tests claim only pointwise asymptotic level.
Take, for instance, the problem of testing whether a distribution has zero mean.
A textbook calculation shows that,
among the distributions with finite variance,
the classical $t$-test
has pointwise asymptotic level $\alpha$ and is pointwise consistent.
Yet \citet{Bahadur1956Dec}
proved that, in that same class of distributions (or various more restricted ones),
any test sequence with uniformly asymptotic level $\alpha$
also satisfies $\limsup_{n \to \infty} Q^n \psi_n \le \alpha$
for every $Q$ with nonzero mean.

It is natural to ask, then:
is testing conditional independence similar to testing for a zero mean?
There are no pointwise-consistent tests with uniformly asymptotic level,
but might there be a pointwise-consistent test with pointwise asymptotic level?

\Citet[Corollary 1]{gyorfiWalk2012} claimed such a test,
specifically a strongly consistent one
-- but their proof was incorrect
\citep[Section 1.1]{neykovEtAl2021}.
The result was stated pointwise,
but the correction (with an added smoothness assumption)
by \citet[Section 5]{gyorfiEtAl2023}
reveals the incorrect proof would have actually implied a strongly consistent test with uniformly asymptotic level, contradicting \cref{thm:sp}.
\Citet{gyorfiEtAl2023} conclude:\footnote{While they frame the problem in terms of testing whether a data transformation is lossless, as they state in their Section 2.2, this is a question of conditional independence; see \cref{sec:gyorfi-walk} for details.}
\begin{quote}
It is an open research
problem whether a strong universal test exists, i.e., a test that is strongly consistent without
any condition [...] on the underlying distribution.
\end{quote}

Proposition 5 of \citet{kci}
would also imply a pointwise-consistent test of conditional independence with pointwise level.
Unfortunately, their proof sketch
is also incorrect.\footnote{The sketch uses a property that holds for the correct conditional kernel mean embeddings, but applies it to the estimated CMEs. The regularization in the regression is also important, but they only state that it is ``small.''}
Asked about this step in personal communication,
one of the authors replied:
\begin{quote}
We also realized later that there is a gap and that it is indeed unclear how to turn the ``proof sketch'' into a proper proof. Any suggestions are welcome!
\end{quote}

\Citet{boekenEtAl2026} study
the possibility of pointwise-consistent tests of conditional independence
using topological criteria.
Their main result is for \emph{finite-precision} tests:
they show that no such test can consistently distinguish the two hypotheses
for all Borel laws on Polish spaces with perfect $\cZ$.
The finite-precision condition is very natural for practical implementation,
but is perhaps not so natural for conditional independence:
the construction of \citet{shahPeters2020}
demonstrates that conditional independence can always be ``hidden'' in lower-order bits of $Z$ than the finite-precision measurements see.
They also note:
\begin{quote}
It remains an open question
whether there exists a consistent [finite-precision] test for conditional independence
if one merely assumes that the distribution has a density.
\end{quote}

We settle all three of these open questions in the negative,
through the following result.

\begin{theorem} \label{thm:main-intro}
  In the setting of \cref{thm:sp},
  suppose $(\psi_n)$ is some family of tests
  with pointwise asymptotic level $\alpha \in [0, 1)$.
  Then $\inf_{Q \in \cE_{M} \setminus \cP_M} \limsup_{n \to \infty} Q^n \psi_n \le \alpha$;
  in particular, $(\psi_n)$ is not pointwise consistent.
\end{theorem}
This result is a special case of \cref{thm:main},
which shows this holds even if all the distributions are additionally required
to have densities bounded by a common constant.
\Cref{thm:regular} strengthens it further:
the level requirement need only hold over conditionally independent laws
with a continuous density and a uniformly continuous conditional law,
and the alternatives may be restricted to $C^\infty$ densities.

The proof operates roughly by constructing a sequence of $P_m \in \cP$
and a conditionally dependent $Q \in \cQ$
such that $P_m^n \to Q^n$ weak-$*$ for each fixed $n$,
and so a fixed $\psi_n$ behaves nearly identically on $P_m^n$ and $Q^n$.
We then build distributions by gluing together diluted copies of these distributions with appropriately budgeted weights.
The distribution $P_*$ which uses $P_m$ at each level
is treated nearly identically to one which uses $Q$ in a single level;
since we must reject with probability at most $\alpha$ on $P_*$,
for every $\varepsilon>0$ there must be such an alternative whose limiting upper rejection rate is at most $\alpha+\varepsilon$.

Compared to \cref{thm:sp},
when switching from uniformly asymptotic level
to pointwise asymptotic level,
\cref{thm:main-intro}'s conclusion is weaker:
it only rules out a positive lower bound on the asymptotic power improvement over $\alpha$.
This is not a weakness in the analysis:
there exist tests with pointwise asymptotic level
powerful against many alternatives,
though not uniformly.
In general, such pointwise tests need only an estimator converging in probability.

\begin{lemma} \label{thm:test-functional}
  Let $\Delta : \cM \to [0, 1]$
  with estimators $\hat\Delta_n \in [0, 1]$
  such that $\hat\Delta_n \to \Delta(\mu)$ in $\mu^n$-probability
  for each $\mu \in \cM$.
  Then for any $\alpha, \delta \in (0, 1)$,
  there is a randomized test sequence $(\psi_n)$
  for $\cP = \{ \mu \in \cM : \Delta(\mu) = 0 \}$
  against $\cQ = \{ \mu \in \cM : \Delta(\mu) > 0 \}$
  with
  \[
    \lim_{n \to \infty} \mu^n \psi_n
    = \alpha + (1-\alpha) \min\{ 1, \Delta(\mu) / \delta \}
    \quad\text{for each } \mu \in \cM
  .\]
  Thus $(\psi_n)$
  has pointwise asymptotic level $\alpha$,
  limiting power strictly above $\alpha$ whenever $\Delta(\mu) > 0$,
  and limiting power one whenever $\Delta(\mu) \ge \delta$.
\end{lemma}
\begin{proof}
  Let $\psi_n = \alpha + (1 - \alpha) \min\{ 1, \hat\Delta_n / \delta \}$.
  The limit follows from the continuity of
  $t \mapsto \min \{1, t / \delta \}$
  and bounded convergence;
  the following implications are immediate.
\end{proof}
If a deterministic test sequence is desired,
this can also be achieved for classes of distributions with a shared atomless one-dimensional projection:
hold out $k_n$ points, with $k_n \to \infty$ and $k_n / n \to 0$,
and use the rank of the first data point's projection among the $k_n-1$ other points
as a source of auxiliary randomization.
This changes each rejection probability by at most $1/(k_n+1)$.

\begin{proposition} \label{thm:ci-functional}
  Let $\cX, \cY, \cZ$ be standard Borel spaces,
  and $\cM$ the set of probability measures on $\cX \times \cY \times \cZ$.
  There exist $\Delta : \cM \to [0, 1]$
  with $\Delta(\mu) = 0$ iff $X \indep Y \mid Z$ under $\mu$,
  as well as estimators $\hat\Delta_n \in [0, 1]$
  such that $\hat\Delta_n \to \Delta(\mu)$ in $\mu^n$-probability
  for each $\mu \in \cM$.
\end{proposition}
Thus, there exist tests for conditional independence
with pointwise asymptotic level
that are pointwise consistent for any $\Delta$-separated distributions;
however, broad classes of distributions have $\inf_{Q \in \cQ} \Delta(Q) = 0$.
Replacing the fixed threshold $\delta$ in \cref{thm:test-functional} by any positive sequence $\delta_n\searrow0$
gives a test consistent at every alternative, since $\hat\Delta_n\to\Delta(Q)>0$ in probability.
But for absolutely continuous distributions on Euclidean spaces,
\cref{thm:single-null} then supplies a single null distribution $P_*$
with $\limsup_n P_*^n\psi_n=1$: %
shrinking the threshold breaks pointwise level control.

\emph{Paper outline.}
We first prove \cref{thm:ci-functional} in \cref{sec:functional-proof},
as it is independent of the rest and can be dispensed with quickly.
\Cref{sec:setup,sec:oscillation,sec:ingredients,sec:main-proof}
then give a self-contained proof of \cref{thm:main},
which generalizes \cref{thm:main-intro}.
\Cref{sec:regular} proves \cref{thm:regular},
which strengthens \cref{thm:main} by using more regular densities.

To situate our results in the literature,
we also provide two alternative proofs of \cref{thm:main-intro}.
In \cref{sec:black-box},
we use the results of \citet{shahPeters2020} as a black box
inside our diagonalization argument.
In \cref{sec:baire},
we replace the explicit diagonalization
with a topological argument following \citet{boekenEtAl2026}:
inside a compact family of laws with a common density bound,
conditionally independent laws are Baire-generic,
and a generic null attains at least the infimum asymptotic rejection rate over the alternatives.

\Cref{sec:discrete-z} proves the positive results for discrete conditioning variables; \cref{sec:gyorfi-walk} applies the impossibility theorem to lossless feature selection as in \citet{gyorfiEtAl2023}.

\section{Proof of the positive result} \label{sec:functional-proof}
We begin by proving \cref{thm:ci-functional}.
There are many such functionals $\Delta$ in common use;
the bigger challenge is an estimator $\hat\Delta_n$
with universal convergence in probability.
(Note that no estimator with \emph{uniform} convergence over Euclidean densities can exist,
as that would yield a test with uniformly asymptotic level and contradict \cref{thm:sp}.)

We use a simple, self-contained partition-based estimate,
similar in spirit to that of \citet{gyorfiWalk2012,gyorfiEtAl2023}.
A more practical approach with equivalent guarantees could be
e.g.\ the KCI functional of \citet{kci}
(also see \citealp{zheng:hardness})
with the universally consistent conditional mean estimator of
\citet{TamasBalazs2024}.

\begin{proof}[Proof of \cref{thm:ci-functional}]
\emph{The functional.}
Let $(\cA_j)_{j\geq1}$ and $(\cB_k)_{k\geq1}$ be countable $\pi$-systems containing $\cX$ and $\cY$, respectively, and generating their Borel $\sigma$-algebras. For instance, we can use countable bases closed under finite intersections, with the whole spaces adjoined.
Take $\mu_{X \mid Z}, \mu_{Y \mid Z}, \mu_{XY\mid Z}$
to be regular conditional laws;
these exist because $\cX, \cY, \cX \times \cY$ are standard Borel.
Let $w_{jk}>0$ with $\sum_{jk} w_{jk}=1$,
and set
\begin{align*}
  d_{jk}(\mu)
  = \E_{Z \sim \mu_Z} \abs[\big]{
    \mu_{XY\mid Z}(\cA_j, \cB_k) - \mu_{X\mid Z}(\cA_j) \mu_{Y\mid Z}(\cB_k)
  }
  ,
  \quad
  \Delta(\mu)&=\sum_{j,k} w_{jk} \,d_{jk}(\mu)
.\end{align*}
Conditional independence under $P$ gives $d_{jk}(P) = 0$, so $\Delta(P)=0$.
If $\Delta(\mu)=0$, we know that except for a $\mu_Z$-null set, $\mu_{XY \mid Z}$ agrees with $\mu_{X \mid Z} \otimes \mu_{Y \mid Z}$ on every rectangle $\cA_j\times\cB_k$. These rectangles form a $\pi$-system generating the product $\sigma$-algebra, hence $X\indep_\mu Y\mid Z$.

\emph{The estimator.}
Because $\cZ$ is standard Borel, there is a Borel isomorphism $\phi$ from $\cZ$ onto a Borel subset of $\R$ \citep[Theorem 15.6]{Kechris1995}, with $\sigma(\phi(Z))=\sigma(Z)$.
Each of
\begin{gather*}
\eta_{jk}^{XY}(z)=\E[\indic_{\cA_j}(X)\indic_{\cB_k}(Y)\mid Z=z],
\\
\eta_j^X(z)=\E[\indic_{\cA_j}(X)\mid Z=z],
\qquad
\eta_k^Y(z)=\E[\indic_{\cB_k}(Y)\mid Z=z]
\end{gather*}
is therefore the regression of a bounded observable response from the real variable $\phi(Z)$.

Partition the $n$ observations, which lie in $\cV = \cX \times \cY \times \cZ$, into
training points $\cV_\tr$ of size $\floor{n/2}$
and evaluation points $\cV_\eval$ of size $\ceil{n/2}$.
Estimate the regressions on $\cV_\tr$ by a partitioning estimate in $\phi(Z)$,
predicting the training mean in each cell $[i h_n,(i+1)h_n)$, $i\in\mathbb Z$, and zero in empty cells,
with $h_n\to0$ and $nh_n\to\infty$.
This estimate is universally consistent for bounded targets on $\R$
in $L^2(\mu_Z)$ \parencite[Theorem~4.2]{gyorfiEtAl2002},
hence also in $L^1(\mu_Z)$.
Thus $\widehat\eta_{jk}^{XY}$, $\widehat\eta_j^X$, $\widehat\eta_k^Y$ converge to their targets in $L^1(\mu_Z)$ in probability under every $\mu$.
Pick $K_n\to\infty$ and set
\[
  \widehat d_{jk}
  =\frac{1}{\abs{\cV_\eval}} \sum_{(x,y,z) \in \cV_\eval}
   \abs*{\widehat\eta_{jk}^{XY}(z) - \widehat\eta_j^X(z) \, \widehat\eta_k^Y(z)},
  \qquad
  \widehat\Delta_n = \sum_{j,k \le K_n} w_{jk}\, \widehat d_{jk}
.\]
Given $\cV_\tr$, each $\widehat d_{jk}$ is an average of $\abs{\cV_\eval}\to\infty$ bounded i.i.d.\ terms, so $\widehat d_{jk} \to \E[\widehat d_{jk}\mid\cV_\tr]$ in probability as $\abs{\cV_\eval}$ grows.
Meanwhile,
using the $L^1(\mu_Z)$ norm,
\[
  \abs[\big]{\E[\widehat d_{jk}\mid\cV_\tr]-d_{jk}(\mu)}
  \le
  \norm{\widehat\eta_{jk}^{XY} - \eta_{jk}^{XY}}%
  + \norm{\widehat\eta_j^X - \eta_j^X}%
  + \norm{\widehat\eta_k^Y - \eta_k^Y}%
  \to0
\]
in probability as $\abs{\cV_\tr}$ grows.
Thus $\widehat d_{jk}\to d_{jk}(\mu)$ in probability for each $(j,k)$, and since the $\widehat d_{jk}$ are bounded and the $w_{jk}$ summable, $\widehat\Delta_n\to\Delta(\mu)$ in probability.
\end{proof}

\section{Setup and main result}
\label{sec:setup}

From now until \cref{sec:discrete-z}, we specialize to $\cV=\R^s$ with $s=d_X+d_Y+d_Z$ and $d_X,d_Y,d_Z\geq1$. A random point of $\R^s$ is $V=(X,Y,Z)$, and a fixed one is $v=(x,y,z)$.
Recall from \cref{thm:sp} that, for $M \in (0, \infty]$,
$\cE_M$ is the class of laws on $\R^s$ that are absolutely continuous with respect to Lebesgue measure and supported in $(-M,M)^s$,
and that $\cP_M \subseteq \cE_M$ consists of those with $X\indep Y\mid Z$;
put $\cQ_M = \cE_M \setminus \cP_M$.

For $\mu \in \cE_\infty$, write $f$ for its Lebesgue density,
$\norm{f}_\infty$ for the essential supremum of $f$,
and $f_Z$, $f_{XZ}$, $f_{YZ}$ for the corresponding marginal densities.
For $L \in (0, \infty]$, let
$\cE_{M,L} \subset \cE_M$ be the distributions whose density $f$ has $\norm f_\infty \le L$,
and
$\cP_{M,L} = \cP_M \cap \cE_{M,L}$,
$\cQ_{M,L} = \cQ_M \cap \cE_{M,L}$.
Using $L=\infty$ recovers $\cE_M$, $\cP_M$, and $\cQ_M$.

A measure $\mu$ on $\cV$ satisfies
conditional independence $X\indep_\mu Y\mid Z$
if and only if
the conditional law of $(X,Y)$ given $Z$ is a product measure for $\mu$-almost every value of $Z$.
If $\mu$ has a density $f$, this is equivalent to
\begin{equation}
  f(x,y,z) \, f_Z(z)
  = f_{XZ}(x,z) \, f_{YZ}(y,z)
  \quad\text{for Lebesgue-almost every }(x,y,z)\in\R^s.
  \label{eq:ci-density}
\end{equation}

The total variation distance between measures $\mu, \nu$
is $\norm{\mu-\nu}_{\TV} = \sup_A \abs{\mu(A) - \nu(A)}$,
where the supremum is over measurable $A$;
this is $\frac12 \norm{f - g}_1$ when $\mu, \nu$ have densities $f, g$.
As tests $\psi_n$ take values in $[0,1]$ and total variation is subadditive under products,
\begin{equation}
      \abs{\mu^n \psi_n - \nu^n \psi_n}
  \le \norm{\mu^n - \nu^n}_{\TV}
  \le n \norm{\mu - \nu}_{\TV}
  \label{eq:rej-tv}
.\end{equation}
It is key to both \citet{shahPeters2020}
and the present work, however,
that this bound is potentially loose:
any alternative has positive total variation distance from all nulls (\cref{thm:tv-closed}).
If $\psi_n$ is Lebesgue-measurable, though,
there are sequences of measures separated in TV on which $\psi_n$ still behaves similarly.

Our main result is the following theorem, proved in \cref{sec:main-proof}.

\begin{theorem}[store=thm:main]
Let $M\in(0,\infty]$ and $L\in\bigl((2M)^{-s},\infty\bigr]$,
interpreting $\infty^{-s}$ as $0$.
For every sequence of tests $(\psi_n)$,
there is a $P_* \in \cP_{M,L}$ with
\[
  \limsup_{n\to\infty} P_*^n \psi_n
  \ge
  \inf_{Q \in \cQ_{M,L}} \limsup_{n\to\infty} Q^n \psi_n
.\]
In particular, if $(\psi_n)$ has pointwise asymptotic level $\alpha$ for the null $\cP_{M,L}$, then
\[
  \inf_{Q\in\cQ_{M,L}}\;\limsup_{n\to\infty}\,Q^n\psi_n\le\alpha,
\]
so for $\alpha<1$ it is not pointwise consistent against $\cQ_{M,L}$.
\end{theorem}

\begin{proof}[Proof of \cref{thm:main-intro}]
Take $L=\infty$ in \cref{thm:main}.
\end{proof}

The common density bound strengthens both sides of the impossibility result: the test needs level only on a smaller null class, yet cannot achieve a positive lower bound on its asymptotic power improvement over the smaller alternative class.
Under the level assumption, for every $\varepsilon>0$ there is a fixed $Q\in\cQ_{M,L}$ with $\limsup_nQ^n\psi_n<\alpha+\varepsilon$.
As established by \cref{thm:test-functional,thm:ci-functional},
however,
this infimum need not actually be attained.

Given an $M$,
the lower bound on $L$ is the best possible:
any distribution on $(-M, M)^s$ must have maximal density at least $(2 M)^{-s}$,
and only a uniform distribution (which is conditionally independent) can achieve that bound.
Thus for any $L \le (2 M)^{-s}$, $\cQ_{M,L}$ is empty.

\begin{corollary} \label{thm:single-null}
In the setting of \cref{thm:main}, if $Q^n\psi_n\to1$ for every $Q\in\cQ_{M,L}$, then $\limsup_nP_*^n\psi_n=1$ for some $P_*\in\cP_{M,L}$.
\end{corollary}

\Cref{thm:single-null} is immediate from \cref{thm:main}.
The following sections build up ingredients used in the proof of the latter.

\section{Oscillating families}
\label{sec:oscillation}
The main component of the proof is a sequence of conditionally independent laws
whose behaviour under any fixed test $\psi_n$
converges to that of a single dependent law.
We first define such families on the unit cube $\cU = [0, 1]^s$.

\begin{definition}
\label{def:oscillating}
Absolutely continuous probability measures $(P_m)_{m \ge 1}$ and $Q$
on $\cU$
are
an \emph{oscillating family with density bound $L<\infty$}
if,
with densities $(p_m)$ and $q$,
\begin{romanenum}
  \item\label{item:of-ci}
  $X \indep_{P_m} Y \mid Z$ for every $m$, and
  $X \nindep_Q Y \mid Z$;
  \item\label{item:of-weakstar} for all $F \in L^1(\cU)$, $P_mF\to QF$,
  or equivalently $p_m\to q$ weak-$*$ in $L^\infty(\cU)$;
  \item\label{item:of-bounds} $\norm{p_m}_\infty\le L$ for every $m$, and $\norm{q}_\infty\le L$.
\end{romanenum}
\end{definition}

Property \labelcref{item:of-weakstar} concerns a single observation.
Together with the density bound, it extends to any fixed number of observations.

\begin{lemma}[Tensorization] \label{thm:tensorization}
Let $(P_m)$ and $Q$ be probability measures on $\cU$
whose densities $p_m$ and $q$ satisfy $\norm{p_m}_\infty \le L$,
$\norm{q}_\infty \le L$,
and $p_m \to q$ weak-$*$ in $L^\infty(\cU)$;
for instance, an oscillating family with density bound $L$.
For every $n \ge 1$ and every $F \in L^1\bigl(\cU^n\bigr)$,
$P_m^n F \to Q^n F$ as $m \to \infty$.
Thus $P_m^n \to Q^n$ setwise, and $P_m^n \psi_n \to Q^n \psi_n$ for each test $\psi_n$.
\end{lemma}
\begin{proof}
If $F(v_{1:n}) = \prod_{i = 1}^n \indic_{A_i}(v_i)$ for measurable $A_i \subseteq \cU$, then by weak-$*$ convergence,
\[
  P_m^nF=\prod_{i\le n}P_m(A_i)\to\prod_{i\le n}Q(A_i)=Q^nF.
\]
Finite linear combinations of such $F$ are dense in $L^1\bigl(\cU^n\bigr)$, so for general $F$ and $\eps>0$ pick such a linear combination $F_\eps$ with $\norm{F - F_\eps}_{L^1(\cU^n)} < \eps$.
From
\[
  \abs{P_m^nF-Q^nF}
  \le
    \abs{P_m^n (F - F_\eps)}
  + \abs{(P_m^n - Q^n) F_\eps}
  + \abs{Q^n (F_\eps - F)}
\]
and the fact that the densities of $P_m^n$ and $Q^n$ are at most $L^n$,
the first and third term are each at most $L^n \eps$,
and the second tends to zero.
As $\eps$ was arbitrary, we have $P_m^n F \to Q^n F$.
Take $F=\indic_A$ for setwise convergence, and $F=\psi_n$ for tests.
\end{proof}

Any oscillating family
must have $\liminf_m \norm{P_m - Q}_\TV > 0$,
as otherwise a subsequence of the conditionally independent $P_m$
would converge to the conditionally dependent $Q$,
contradicting \cref{thm:tv-closed}.
Setwise convergence is still possible,
because the sets attaining total variation's supremum over measurable sets
change with $m$:
no fixed test follows the oscillation to arbitrarily fine scales.
The families we use oscillate in the first coordinate of $Z$;
for them, property \labelcref{item:of-weakstar} reduces to a one-dimensional Riemann--Lebesgue statement.

\begin{lemma}
\label{thm:oscillation-1d}
Let $g_m\in L^\infty([0,1])$ satisfy $\norm{g_m}_\infty\leq1$ and
$\sup_{0\le a<b\leq1}\bigl|\int_a^b g_m(t)\ud t\bigr|\to0$ as $m\to\infty$.
Then for every $G\in L^1(\cU)$,
$\int_{\cU} G(v)\,g_m(z^{(1)})\ud v\to0$.
\end{lemma}

\begin{proof}
For $G$ the indicator of a box whose $z^{(1)}$ side is $(a,b)$,
the integral is $\int_a^b g_m(t) \ud t$ times the volumes of the remaining sides,
which are each at most one.
Hence the integral's absolute value is at most $\sup_{a<b}\bigl|\int_a^b g_m(t) \ud t \bigr| \to 0$.
Finite linear combinations of box indicators are dense in $L^1$, and $\norm{g_m}_\infty\leq1$, so the same approximation argument as in the proof of \cref{thm:tensorization} extends the conclusion to every $G\in L^1(\cU)$.
\end{proof}

\begin{lemma}[Piecewise constant family]
\label{thm:oscillation}
For each $\eta\in(0,1]$, there is an oscillating family
$((P^\eta_m)_{m\geq1}, Q^\eta)$ with density bound $(1+\eta)^2$
in which
$\norm{P^\eta_m-Q^\eta}_{\TV}=\eta/2$ for every $m$.
\end{lemma}
\begin{proof}
Define $P_m^\eta$ and $Q^\eta$ through the densities
(illustrated in \cref{fig:pm-q})
\[
  p_m^\eta(v) = \bigl( 1 + \eta r_m(z) a_X(x) \bigr) \bigl( 1 + \eta r_m(z) a_Y(y) \bigr)
  \quad\text{and}\quad
  q^\eta(v) = 1 + \eta^2 a_X(x) a_Y(y)
,\]
given in terms of the $\{-1, 1\}$-valued functions
\[
  a_X(x) = 2 \indic(x^{(1)}<1/2) - 1
  ,\quad
  a_Y(y) = 2 \indic(y^{(1)}<1/2) - 1
  ,\quad
  r_m(z)=(-1)^{\floor{ 2m z^{(1)} }}
.\]
\begin{figure}[t]
  \centering
  \begin{tikzpicture}[x=1cm,y=1cm]
  \begin{scope}
    \celltwo{0}{0}\cellzero{1}{0}\cellzero{0}{1}\celltwo{1}{1}
    \panelframe{$Q^1$}
  \end{scope}

  \begin{scope}[shift={(3,0)}]
    \cellfour{0}{0}\cellzero{1}{0}\cellzero{0}{1}\cellzero{1}{1}
    \panelframe{$P_m^1 \mid r_m(z) = +1$}
  \end{scope}

  \begin{scope}[shift={(6,0)}]
    \cellzero{0}{0}\cellzero{1}{0}\cellzero{0}{1}\cellfour{1}{1}
    \panelframe{$P_m^1 \mid r_m(z) = -1$}
  \end{scope}

  \def\W{3}

  \begin{scope}[shift={(9,1.6)}]
    \fill[rpos] (0,0) rectangle ({\W/2},.5);
    \fill[rneg] ({\W/2},0) rectangle (\W,.5);
    \node[font=\scriptsize] at ({\W/4},.25) {$+$};
    \node[font=\scriptsize,text=white] at ({3*\W/4},.25) {$-$};
    \draw[line width=.7pt] (0,0) rectangle (\W,.5);
    \foreach \t/\l in {0/{$0$},0.5/{$\tfrac12$},1/{$1$}}{%
      \draw[line width=.5pt] ({\t*\W},0)--({\t*\W},-.09);
      \node[font=\scriptsize,below=0pt] at ({\t*\W},-.09) {\l};}
    \node[font=\scriptsize] at ({\W/4},-.38) {$z^{(1)}$};
    \node[font=\small] at ({\W/2},.92) {$r_1(z)=(-1)^{\lfloor 2z^{(1)}\rfloor}$};
  \end{scope}

  \begin{scope}[shift={(9,-0.3)}]
    \foreach \k in {0,2,4}{%
      \fill[rpos] ({\k*\W/6},0) rectangle ({(\k+1)*\W/6},.5);
      \node[font=\scriptsize] at ({(\k+.5)*\W/6},.25) {$+$};}
    \foreach \k in {1,3,5}{%
      \fill[rneg] ({\k*\W/6},0) rectangle ({(\k+1)*\W/6},.5);
      \node[font=\scriptsize,text=white] at ({(\k+.5)*\W/6},.25) {$-$};}
    \draw[line width=.7pt] (0,0) rectangle (\W,.5);
    \foreach \t/\l in {0/{$0$},0.5/{$\tfrac12$},1/{$1$}}{%
      \draw[line width=.5pt] ({\t*\W},0)--({\t*\W},-.09);
      \node[font=\scriptsize,below=0pt] at ({\t*\W},-.09) {\l};}
    \foreach \k in {1,2,4,5}{%
      \draw[line width=.4pt,black!45] ({\k*\W/6},0)--({\k*\W/6},-.05);}
    \node[font=\scriptsize] at ({\W/4},-.38) {$z^{(1)}$};
    \node[font=\small] at ({\W/2},.92) {$r_{3}(z)=(-1)^{\lfloor 6z^{(1)}\rfloor}$};
\end{scope}
\end{tikzpicture}
  \caption{The construction of $P_m^1$ and $Q^1$ in \cref{thm:oscillation}. Each marginal of $(X, Y)$ is the same; the $P_m^1$ differ only in how $Z$ is split.
  For $\eta \searrow 0$, the densities are constant in the same blocks,
but each value approaches $1$.
\label{fig:pm-q}}
\end{figure}
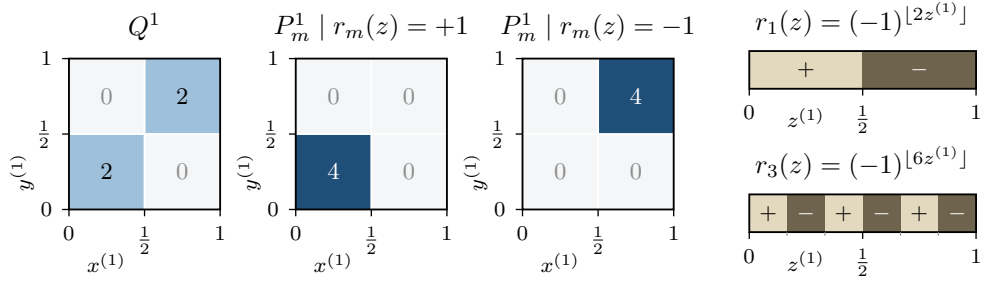
Since $a_X$ and $a_Y$ integrate to zero, $p_m^\eta$ and $q^\eta$ integrate to one.
We also have $(1-\eta)^2 \le p_m^\eta(v) \le (1 + \eta)^2$
and $1 - \eta^2 \le q^\eta(v) \le 1 + \eta^2 < (1 + \eta)^2$,
giving property \labelcref{item:of-bounds}.

For fixed $z$, $p_m^\eta$ is either $(1 + \eta a_X)(1 + \eta a_Y)$ or $(1 - \eta a_X)(1 - \eta a_Y)$, a product, so $X\indep_{P_m^\eta}Y\mid Z$.
Under $Q$, the conditional marginals of $X$ and $Y$ are $1$,
while the conditional joint density is $1 + \eta^2 a_X a_Y$.
Since $\abs{a_X a_Y} = 1$ and $\eta > 0$, $X\nindep_QY\mid Z$,
and we have \labelcref{item:of-ci}.

Using $r_m^2=1$,
\begin{equation}
  p_m^\eta(v) - q^\eta(v) = \eta \, r_m(z) \bigl( a_X(x) + a_Y(y) \bigr).
  \label{eq:pm-minus-q}
\end{equation}
Notice $\abs{a_X+a_Y}$ equals $2$ on the set $\{a_X=a_Y\}$, which has measure $1/2$, and vanishes elsewhere;
thus $2 \norm{P_m^\eta - Q^\eta}_{\TV} = \int \abs{p_m^\eta - q^\eta} \ud v
= \int \eta \abs{a_X + a_Y} \ud x \ud y = \eta$.

\Cref{eq:pm-minus-q} gives $(P_m^\eta - Q^\eta) F = \int G(v)\,r_m(z)\ud v$ 
with $G = F \cdot \eta (a_X + a_Y) \in L^1(\cU)$.
Antiderivatives of $g_m(t) = (-1)^{\floor{2 m t}}$
are piecewise-linear and oscillate in an interval of length $1/(2m) \to 0$,
so \cref{thm:oscillation-1d} applies and shows \labelcref{item:of-weakstar}.
\end{proof}

\section{Gluing and local replacement}
\label{sec:ingredients}

We now place rescaled copies of an oscillating family inside one distribution.
A \emph{box} is a product $B=\prod_{i\le s}(a_i,b_i)$ of bounded open intervals in $\R^s$;
we use $\vol(B)$ for its Lebesgue measure,
and $B_Z$ for its projection onto the $Z$ coordinates.
Let $T_B : \cU \to \overline B$ be the affine bijection $T_B(u)=(a_i+(b_i-a_i)u_i)_{i\le s}$.
If $\mu$ is a probability measure on $\cU$ with density $f$, then $(T_B)\push\mu$ is concentrated on $B$, has support in $\overline B$, and has density $f(T_B^{-1}v)/\vol(B)$ on $B$.
Since $T_B$ acts separately and bijectively on the blocks $x$, $y$, and $z$, $X\indep Y\mid Z$ holds under $(T_B)\push\mu$ if and only if it holds under $\mu$.

Mixing components that live over disjoint $Z$-regions preserves their conditional structure, because conditioning on $Z$ reveals which component was drawn.

\begin{lemma}
\label{thm:gluing}
Let $I_0,I_1,\dots$ be pairwise-disjoint Borel subsets of $\R^{d_Z}$, and for each $j$ let $R_j\in\cE_\infty$ satisfy $R_j(\R^{d_X+d_Y}\times I_j)=1$.
Let $w_j\geq0$ with $\sum_j w_j=1$, and put $\mu = \sum_j w_j R_j$.
Then $\mu \in \cE_\infty$, and $X\indep_\mu Y\mid Z$ if and only if
for every $j$ with $w_j > 0$, $X \indep_{R_j} Y \mid Z$.
\end{lemma}

\begin{proof}
By countable additivity, $\mu$ is absolutely continuous with density $f = \sum_j w_j r_j$.
Write $E_j = \R^{d_X+d_Y} \times I_j$.
Because $R_j$ is concentrated on $E_j$, the density $r_j$ and the marginal densities $(r_j)_Z$, $(r_j)_{XZ}$, and $(r_j)_{YZ}$ vanish on almost all points whose $z$ coordinate is outside $I_j$.
As the $I_j$ are disjoint,
almost everywhere on $E_j$ we therefore have
$f = w_j r_j$,
$f_Z = w_j \, (r_j)_Z$,
$f_{XZ} = w_j \, (r_j)_{XZ}$,
and $f_{YZ} = w_j \, (r_j)_{YZ}$;
each side of \eqref{eq:ci-density} for $f$ is $w_j^2$ times the corresponding side for $r_j$.
Off $\bigcup_{w_j>0} E_j$, both sides for $f$ vanish almost everywhere.
Hence \eqref{eq:ci-density} holds for $f$ if and only if, for every $j$ with $w_j>0$, it holds for $r_j$ almost everywhere on $E_j$.
For $r_j$, this is the same as holding almost everywhere on $\R^s$.
\end{proof}

\begin{lemma}
\label{thm:local-replacement}
Let $B$ be a box, and $((P_m), Q)$ an oscillating family with density bound $L$.
\begin{romanenum}
  \item\label{item:lr-basic} $(T_B) \push P_m$ and $(T_B) \push Q$ are absolutely continuous and concentrated on $B$, with supports in $\overline B$ and densities bounded by $L/\vol(B)$.
  Moreover, $X\indep_{(T_B) \push P_m}Y\mid Z$ for every $m$,
  while $X\nindep_{(T_B) \push Q}Y\mid Z$.
  \item\label{item:lr-conv} Let $\mu$ be a finite nonnegative measure on $\R^s$ and $w\geq0$ with $\mu(\R^s)+w=1$.  For every $n\geq1$ and every test $\psi_n$,
        \[
          \bigl(\mu + w\,(T_B)\push P_m\bigr)^n\psi_n\longrightarrow\bigl(\mu + w\,(T_B)\push Q\bigr)^n\psi_n
          \qquad\text{as }m\to\infty.
        \]
\end{romanenum}
\end{lemma}

\begin{proof}
  Applying the properties of $(T_B)\push$ stated above
  to \cref{def:oscillating}
  gives part \labelcref{item:lr-basic}.

For part \labelcref{item:lr-conv}, write $S_m = (T_B) \push P_m$ and $S = (T_B) \push Q$, and let $F$ be a bounded measurable function on $(\R^s)^n$.
Expanding out the product measure
in terms of subsets of coordinates,
\[
  (\mu+wS_m)^nF
  =\sum_{A \subseteq [n]} w^{\abs A} \, S_m^{\abs A} G_A
  \quad\text{for}\;\;
  G_A(v_A) = \int F(v_{1:n}) \ud \mu^{\abs{A^c}}(v_{A^c}),
\]
where $v_A$ collects the coordinates indexed by $A$.
The same identity holds with $S$ in place of $S_m$.
For nonempty $A$, each $G_A$ is bounded, so $G_A\circ(T_B\times\dots\times T_B)$ is in $L^1$ on $\cU^{\abs A}$, and \cref{thm:tensorization} gives $S_m^{\abs A}G_A\to S^{\abs A}G_A$.
The term for the empty $A = \{\}$ is $G_{\{\}} = \mu^nF$, which does not depend on $m$.
Thus each of the finitely many terms in the expansion converges to the corresponding term for $S$,
giving $(\mu + w S_m)^n F \to (\mu + w S)^n F$.
Take $F = \psi_n$.
\end{proof}

\section{Proof of the main theorem}
\label{sec:main-proof}

At each stage we insert a dependent component of positive mass, choose a sample size at which the test rejects it sufficiently often, and replace it by a null component with nearly the same rejection probability.
The masses assigned to later stages are then made small enough to preserve that probability in the final distribution.
The following formulation allows a different oscillating family at each stage, as will be needed in the extension of \cref{sec:regular}.

\begin{proposition}[Diagonal construction]
\label{thm:diagonal}
Let $M \in (0, \infty]$,
$G$ be a box with $\overline G \subset (-M,M)^s$,
and $\theta \in (0,1)$.
Write $(a, a+\ell)$ for the first $Z$-coordinate interval of $G$,
and put $b_k = a + \bigl[1-(1-\theta)2^{-k}\bigr] \ell$ for $k \ge 0$.
Slice $G$ from left to right into the boxes
\begin{align*}
     B_0
  &= \bigl\{ v\in G : a < z^{(1)} < b_0 \bigr\},
\\   B_k
  &= \bigl\{ v\in G : b_{k-1} < z^{(1)} < b_k \bigr\}
     \quad\text{for } k \ge 1
.\end{align*}
Note $b_0 = a + \theta \ell$ and $b_k \nearrow a + \ell$,
so $B_0$ takes up a fraction $\theta$ of $G$
and each subsequent slab takes half of the remaining length.

Fix $L \in [1/\theta, \infty)$,
$L_0 = L / \vol(G)$,
and let $U$ be a conditionally independent probability measure on $\cU = [0,1]^s$ with density $u$ having $\norm u_\infty \le \theta L$.

For each $k \ge 1$, let $((P_{k,m})_{m \ge 1}, Q_k)$ be an oscillating family with density bound $L$.

Let $(\psi_n)$ be a sequence of tests.

There are sample sizes $n_1<n_2<\cdots$, indices $m_1,m_2,\dots$,
and positive weights $w_k$ with
$\sum_{k \ge 0} w_k = 1$,
$w_k \le 2(1-\theta)4^{-k}$ for $k \ge 1$,
such that
\[
  P_*=w_0\,(T_{B_0})\push U+\sum_{k\geq1}w_k\,(T_{B_k})\push P_{k,m_k}
\]
lies in $\cP_{M,L_0}$ and satisfies
\[
  \liminf_{k\to\infty}P_*^{n_k}\psi_{n_k}\ge\inf_{Q\in\cQ_{M,L_0}}\limsup_{n\to\infty}Q^n\psi_n.
\]
\end{proposition}

\begin{proof}
The slabs $B_k$ have pairwise disjoint $(B_k)_Z$ and cover $G$ up to a null set, with $\vol(B_0)=\theta\vol(G)$ and $\vol(B_k) = (1-\theta)2^{-k}\vol(G)$ for $k\ge1$.

Let $R_0=(T_{B_0})\push U$, whose density is at most $\theta L/\vol(B_0)=L_0$.
For $k, m \ge 1$ let
$R_{k,m} = (T_{B_k}) \push P_{k,m}$ and
$D_k = (T_{B_k}) \push Q_k$.

Call $\mu$ a \emph{budgeted slab mixture}
if it has the form
\begin{align}
  \mu = \sum_{k \ge 0} w_k S_k,
  \quad
  \text{with }&
  S_0 = R_0,
  \quad
  &S_k \in \{D_k\} \cup \{R_{k,m} : m \ge 1\}
  \;\;\text{for } k \ge 1,
  \label{eq:mixture-form}
\\
  &w_k \ge 0,
  & \sum_k w_k = 1, \qquad
  w_k \le (1-\theta)2^{-k} \;\;\text{for } k \ge 1
  \label{eq:weight-constraint}
.\end{align}
Every budgeted slab mixture
is absolutely continuous with support in $\overline G$.
Since the slabs are disjoint,
the mixture's density on $B_k$ is $w_k$ times that of $S_k$;
for $k = 0$ this is at most $L_0$.
Noting that $\vol(B_k) = \frac{b_k - b_{k-1}}{\ell} \vol(G) = (1 - \theta) (2^{-k+1} - 2^{-k}) \vol(G)$
for $k \ge 1$,
the density is at most $(1-\theta)2^{-k} L / \vol(B_k) = L_0$
by \cref{thm:local-replacement}\labelcref{item:lr-basic}.
Thus a budgeted slab mixture is in $\cE_{M,L_0}$.
By \cref{thm:gluing} with $I_k=(B_k)_Z$, it lies in $\cP_{M,L_0}$ if every $S_k$ with $k\geq1$ and $w_k>0$ is one of the $R_{k,m}$, and in $\cQ_{M,L_0}$ if $S_k=D_k$ for some $k\geq1$ with $w_k>0$.

\emph{Recursion.}
Write $\rho = \inf_{Q\in\cQ_{M,L_0}} \limsup_n Q^n \psi_n$,
and fix a sequence $\lambda_k \searrow 0$.

Set $n_0 = 0$ and $w_1 = (1 - \theta) / 2$.

At stage $k\geq1$ we are given $n_{k-1}$, positive weights $w_1,\dots,w_k$ satisfying the bounds $w_j\le(1-\theta)2^{-j}$ of \eqref{eq:weight-constraint}, and null components $R_j=R_{j,m_j}$ for $1\le j<k$.
Let
\begin{equation}
  \mu_k = \sum_{j=1}^{k-1} w_j R_j + \Bigl(1 - \sum_{j=1}^k w_j\Bigr) R_0
  ,\qquad
  \widetilde Q_k = \mu_k + w_k D_k
  \label{eq:stage-definitions}
.\end{equation}
The measure $\mu_k$ collects the components fixed so far and has total mass $1 - w_k$;
the budgeted slab mixture $\widetilde Q_k$ activates the dependent component in $B_k$ and, as $w_k>0$, lies in $\cQ_{M,L_0}$.
By the definition of $\rho$, $\limsup_n \widetilde Q_k^{n} \psi_n \ge \rho$,
so we can choose $n_k > n_{k-1}$ with
$
  \widetilde Q_k^{n_k} \psi_{n_k} \ge \rho - \lambda_k
$.
By \cref{thm:local-replacement}\labelcref{item:lr-conv}
with $\mu = \mu_k$ and $w = w_k$,
we can then choose $m_k$ such that $R_k = R_{k,m_k}$
and $\widetilde P_k = \mu_k + w_k R_k$ satisfy
\begin{equation}
  \widetilde P_k^{n_k} \psi_{n_k}
  \ge \widetilde Q_k^{n_k} \psi_{n_k} - \lambda_k
  \ge \rho - 2 \lambda_k
  \label{eq:null-rejection}
.\end{equation}
Finally set
\begin{equation}
  w_{k+1}=\min\Bigl\{\frac{w_k}4,\frac{\lambda_k}{2n_k}\Bigr\},
  \label{eq:tail-budget}
\end{equation}
for which indeed $0 < w_{k+1} \le (1 - \theta) 2^{-(k+1)}$.
We can then pass to stage $k+1$.

\emph{Push it to the limit.}
Since from \eqref{eq:tail-budget}
we have $w_{j+1} \le w_j/4$ for every $j \ge 1$,
we know $\sum_{j>k} w_j \le \frac43 w_{k+1}$ for every $k \ge 0$.
For $k=0$ this gives $\sum_{j\geq1}w_j\leq2(1-\theta)/3$, and for $k\geq1$, \eqref{eq:tail-budget} gives
\begin{equation}
  \sum_{j>k}w_j\leq2w_{k+1}\le\frac{\lambda_k}{n_k}.
  \label{eq:weight-tail}
\end{equation}
Put $w_0 = 1 - \sum_{j\ge1} w_j \ge 1 - 2 (1-\theta)/3 > 0$.
Take $P_* = \sum_{j \ge 0} w_j R_j$, a budgeted slab mixture whose active components are all nulls; thus $P_* \in \cP_{M,L_0}$.
The stage-$k$ law is $\widetilde P_k=\sum_{j=1}^k w_jR_j+\bigl(1-\sum_{j=1}^k w_j\bigr)R_0$, so $P_*-\widetilde P_k=\sum_{j>k}w_j\,(R_j-R_0)$. By \eqref{eq:weight-tail} and the fact that total variation between probability measures is at most one, we have
\[
  \norm{P_* - \widetilde P_k}_{\TV}
  \le \sum_{j>k} w_j \norm{R_j - R_0}_{\TV}
  \le \frac{\lambda_k}{n_k}
.\]
Together with \eqref{eq:rej-tv} and \eqref{eq:null-rejection},
\[
  P_*^{n_k}\psi_{n_k}
  \ge \widetilde P_k^{n_k} \psi_{n_k} - n_k \norm{P_*-\widetilde P_k}_{\TV}
  \ge \rho - 3 \lambda_k
;\]
letting $k\to\infty$ gives $\liminf_k P_*^{n_k} \psi_{n_k} \ge \rho$.
\end{proof}

The same proof allows the infimum to be taken over a smaller alternative class, provided that class contains every possible stage law $\widetilde Q_k$ in \eqref{eq:stage-definitions}.
If we define $\rho$ using that smaller class and run the recursion with this value, nothing else changes.
This observation will allow us to restrict the alternatives to smooth densities in \cref{sec:regular}.

We now have all the ingredients to prove the main result.
\getkeytheorem{thm:main}
\begin{proof}
Choose a box $G$ with $\overline G \subset (-M, M)^s$
and an $\eta \in (0, 1]$
such that $(1 + \eta)^2 / \vol(G) \le L$;
this is always possible,
since we may take $\vol(G)$ arbitrarily close to $(2 M)^s$
and we have $L > (2 M)^{-s}$.
For $M = \infty$, it suffices to take $G$ large.

Now apply \cref{thm:diagonal} with
$L_\eta = (1+\eta)^2$,
$\theta = (1+\eta)^{-2}$,
$U$ the uniform distribution on $\cU$ (whose density is $1 = \theta L_\eta$),
and the family of \cref{thm:oscillation} at every stage.
This gives a density bound $L_\eta / \vol(G) \le L$,
together with $P_* \in \cP_{M,L_\eta/\vol(G)} \subseteq \cP_{M,L}$
and sample sizes $n_k$ such that
\[
      \liminf_k P_*^{n_k} \psi_{n_k}
  \ge \inf_{Q \in \cQ_{M,L_\eta/\vol(G)}} \limsup_n Q^n \psi_n
  \ge \inf_{Q \in \cQ_{M,L}} \limsup_n Q^n \psi_n
.\]
Clearly
$\limsup_n P_*^{n} \psi_{n} \ge \liminf_k P_*^{n_k} \psi_{n_k}$.
If $(\psi_n)$ has pointwise asymptotic level $\alpha$,
the left-hand side is at most $\alpha$;
if it is pointwise consistent, the right-hand side is $1$.
\end{proof}

The diagonalization places no constraint on the growth of $n_k$.
A test may need a very large sample to detect a component of mass $w_k$; once that sample size is found, \eqref{eq:tail-budget} makes the total mass of later components negligible at that sample size.

\section{A regular witness}
\label{sec:regular}

The witness $P_*$ of \cref{thm:main} has a discontinuous density,
and its conditional law jumps between two product measures
at frequencies $m_k$ depending on the test.
(Positive results typically assume some control of the variation of the conditional laws in $z$.)
Neither feature is essential:
using a smooth oscillating family in \cref{thm:diagonal}, with amplitudes decreasing to zero, gives a continuous density and a uniformly continuous conditional law.
The rate of continuity, however, is not uniform over the family.

\begin{lemma}[Smooth family]
\label{thm:smooth-oscillation}
For every $C>1$, there exist a product probability measure $U$ on $\cU=[0,1]^s$,
product probability measures $(\Lambda_t)_{t\in[-1,1]}$ on $[0,1]^{d_X+d_Y}$,
and, for each $\eps\in(0,1]$,
an oscillating family $((P^\eps_m),Q^\eps)$ with density bound $C (1+\eps)^2$,
with the following properties.
\begin{romanenum}
  \item\label{item:sf-density} The laws $U$, $P^\eps_m$, and $Q^\eps$ have $C^\infty$ densities with compact support in $(0,1)^s$.
  The density of $U$ is bounded by $C$.
  \item\label{item:sf-kernel} Versions of the conditional law of $(X,Y)$ given $Z=z$ under $U$ and $P^\eps_m$ are, respectively,
  $\Lambda_0$
  and
  $\Lambda_{\eps\sin(2\pi m z^{(1)})}$.
  For all $t,t'\in[-1,1]$,
  \begin{equation}
    \norm{\Lambda_t-\Lambda_{t'}}_{\TV}\le 2\abs{t-t'}.
    \label{eq:smooth-tv}
  \end{equation}
  \item\label{item:sf-smooth-kernel} The versions in \labelcref{item:sf-kernel}, and a version of the conditional law under $Q^\eps$,
  have densities that are $C^\infty$, with bounded derivatives of every order for each fixed $\eps,m$.
  Their supports in the $(x,y)$ coordinates lie in a common compact subset of $(0,1)^{d_X+d_Y}$.
\end{romanenum}
\end{lemma}
\begin{proof}
Choose a $C^\infty$ probability density $\beta_1$ with compact support in $(0,1)$,
symmetric about $1/2$, and satisfying $\norm{\beta_1}_\infty \le C^{1/s}$,
by normalizing symmetric smooth cutoffs that approach one on $(0,1)$.
Define
a $C^\infty$ probability density with compact support in $(0,1)^s$
by
$\beta(v)=\beta_X(x)\beta_Y(y)\beta_Z(z)$, where
\[
  \beta_X(x)=\prod_{i = 1}^{d_X}\beta_1(x^{(i)}),
  \qquad
  \beta_Y(y)=\prod_{i = 1}^{d_Y}\beta_1(y^{(i)}),
  \qquad
  \beta_Z(z)=\prod_{i = 1}^{d_Z}\beta_1(z^{(i)})
.\]
Let $U$ be the product law with density $\beta$; its density is bounded by $\norm{\beta_1}_\infty^s\le C$.

Let $a_X(x) = \sin(2\pi x^{(1)})$ and $a_Y(y) = \sin(2\pi y^{(1)})$.
For $t \in [-1, 1]$, let $\Lambda_t$ be the probability measure on $[0,1]^{d_X+d_Y}$ with density
\begin{equation}
  \lambda_t(x,y)=\beta_X(x)\bigl(1+t\,a_X(x)\bigr)\,\beta_Y(y)\bigl(1+t\,a_Y(y)\bigr)
  \label{eq:Lambda}
.\end{equation}
Because $\beta_1$ is symmetric about $1/2$ and $\sin(2\pi\,\cdot)$ is antisymmetric about $1/2$,
we have
$\int \beta_X a_X \ud x = \int_0^1 \beta_1(t) \sin(2\pi t) \ud t=0$;
likewise, $\int \beta_Y a_Y \ud y = 0$.
Hence for $\abs{t} \le 1$, the nonnegative functions $\beta_X(1+ta_X)$ and $\beta_Y(1+ta_Y)$ are probability densities,
and $\Lambda_t$ is a product probability measure.

Let $r_m(z)=\sin(2\pi m z^{(1)})$,
and define
\begin{align*}
     p^\eps_m(v)
  &= \beta(v) \bigl( 1 + \eps \, r_m(z) a_X(x) \bigr)
     \bigl( 1 + \eps \, r_m(z) a_Y(y) \bigr)
   = \beta_Z(z) \, \lambda_{\eps r_m(z)}(x, y)
,\\
     q^\eps(v)
  &= \beta(v) \Bigl( 1 + \frac{\eps^2}{2} a_X(x) a_Y(y) \Bigr)
   = \beta_Z(z) \lambda_0(x, y)
     \Bigl( 1 + \frac{\eps^2}{2} a_X(x) a_Y(y) \Bigr)
.\end{align*}
As $\eps^2 \le 1$,
each of these is nonnegative,
and integrates to one.
Thus they are $C^\infty$ probability densities
with the same support as $\beta$.
That $X \indep_{P^\eps_m} Y \mid Z$ is immediate from the form of $\lambda_t$;
under $Q^\eps$, the conditional marginals have densities $\beta_X$ and $\beta_Y$,
while their product differs from the conditional joint density on
the positive-measure set where $\beta_X a_X \beta_Y a_Y\ne0$.
Thus $X\nindep Y\mid Z$ under $Q^\eps$,
and we have \cref{def:oscillating}\labelcref{item:of-ci}.

Also, $\norm{p_m^\eps}_\infty \le (1 + \eps)^2 \norm{\beta}_\infty$,
and $\norm{q^\eps}_\infty \le (1 + \eps^2 / 2) \norm{\beta}_\infty \le (1+\eps)^2 \norm\beta_\infty$.
Since $\norm{\beta}_\infty = \norm{\beta_1}_\infty^s\le C$,
this shows property \labelcref{item:sf-density}
and \cref{def:oscillating}\labelcref{item:of-bounds}.

The factorizations $\beta(v)=\beta_Z(z)\lambda_0(x,y)$ and $p^\eps_m(v)=\beta_Z(z)\lambda_{\eps r_m(z)}(x,y)$ give the conditional laws in property \labelcref{item:sf-kernel}.
The other part of \labelcref{item:sf-kernel}, the TV bound, follows from
\begin{align*}
       \norm{\Lambda_t-\Lambda_{t'}}_{\TV}
  &  = \frac{\abs{t-t'}}2 \int \beta_X(x) \beta_Y(y) \,\bigl|
          a_X(x) + a_Y(y)
        + (t+t') a_X(x) a_Y(y)
        \bigr| \ud x \ud y
\\&\le \frac{\abs{t-t'}}{2} \bigl( 2 + \abs{t+t'} \bigr)
   \le 2 \abs{t - t'}
.\end{align*}

Under $Q^\eps$, the conditional density can be taken to be
$\lambda_0(x,y)\bigl(1+\frac{\eps^2}{2}a_X(x)a_Y(y)\bigr)$, independently of $z$.
This density, $\lambda_0(x,y)$, and $\lambda_{\eps r_m(z)}(x,y)$ are $C^\infty$ in $(x,y,z)$,
with bounded derivatives of every order for each fixed $\eps,m$,
and all are supported in $(\supp\beta_1)^{d_X+d_Y}$ in the $(x,y)$ coordinates.
This proves property \labelcref{item:sf-smooth-kernel}.

It remains to show \cref{def:oscillating}\labelcref{item:of-weakstar}.
Using $r_m^2=\frac12(1-c_m)$ with $c_m(z)=\cos(4\pi mz^{(1)})$,
\[
  p^\eps_m - q^\eps
  = \beta \Bigl[
    \eps \, (a_X + a_Y) \, r_m
  - \frac{\eps^2}{2} \, a_X a_Y \, c_m
  \Bigr]
.\]
For $F\in L^1(\cU)$, therefore, with $G_1=F\beta(a_X+a_Y)$ and $G_2=F\beta a_Xa_Y$ in $L^1(\cU)$,
\[
  P^\eps_mF-Q^\eps F=\eps\int G_1r_m-\frac{\eps^2}2\int G_2c_m.
\]
Both terms tend to zero as $m \to \infty$ by \cref{thm:oscillation-1d}, since for all $0\le a<b\leq1$
\[
  \Bigl|\int_a^b\sin(2\pi mt)\ud t\Bigr|\le\frac1{\pi m},
  \qquad
  \Bigl|\int_a^b\cos(4\pi mt)\ud t\Bigr|\le\frac1{2\pi m}
.\qedhere \]
\end{proof}

The parameter $\eps$ controls the amplitude of the conditional oscillation: the conditional laws of $P^\eps_m$ stay within $2\eps$ of the fixed product $\Lambda_0$ in total variation, uniformly in $m$ and $z$, while the frequency $m$ remains free.

\begin{theorem}
\label{thm:regular}
Let $M \in (0, \infty]$
and $L\in\bigl((2M)^{-s},\infty\bigr]$, interpreting $\infty^{-s}$ as $0$.
For every sequence of tests $(\psi_n)$
there is $P_* \in \cP_{M,L}$ with
\[
  \limsup_{n\to\infty}P_*^n\psi_n\ge\inf_{Q\in\cQ_{M,L}}\limsup_{n\to\infty}Q^n\psi_n,
\]
and moreover
\begin{romanenum}
  \item\label{item:reg-density} the density of $P_*$ is continuous on $\R^s$, and $C^\infty$ off a hyperplane $\{z^{(1)}=c\}$ for some $c\in\R$;
  \item\label{item:reg-kernel} the conditional law of $(X,Y)$ given $Z$ under $P_*$ has a version $z\mapsto\kappa^*(z)$ that is uniformly continuous in total variation on $\R^{d_Z}$.
\end{romanenum}
In particular, a test sequence with pointwise asymptotic level $\alpha$ merely over the laws in $\cP_{M,L}$ having these two properties still satisfies $\inf_{Q\in\cQ_{M,L}}\limsup_nQ^n\psi_n\le\alpha$.
Both conclusions remain valid when the alternative class is restricted to laws with $C^\infty$ densities.
\end{theorem}
\begin{proof}
Choose $C>1$, $\eps_0\in(0,1]$, and a box $G$ with $\overline G\subset(-M,M)^s$ such that
\[
  L':=\frac{C(1+\eps_0)^2}{\vol(G)}\le L.
\]
This is possible by taking $C$ close to one, $\eps_0$ close to zero, and $\vol(G)$ close to $(2M)^s$;
for $M=\infty$, take $G$ sufficiently large.
Write $(a,a+\ell)$ for the first $Z$-coordinate interval of $G$.
Apply \cref{thm:smooth-oscillation} with this $C$ to obtain $U$, $(\Lambda_t)$, and the oscillating families for all $\eps\in(0,1]$.
Put $\eps_k=\eps_0 2^{-k}$, $H=C(1+\eps_0)^2$, and $\theta=(1+\eps_0)^{-2}$.
By \cref{thm:smooth-oscillation}\labelcref{item:sf-density}, the density of $U$ is bounded by $C=\theta H$,
and the family $((P^{\eps_k}_m),Q^{\eps_k})$ has density bound at most $H$.
Since $U$ is a product law, \cref{thm:diagonal} applies with density bound $H$ and these families at the respective stages.
It gives $P_*\in\cP_{M,L'}\subseteq\cP_{M,L}$ and sample sizes $n_k$ with
\[
  \liminf_k P_*^{n_k}\psi_{n_k}
  \ge\inf_{Q\in\cQ_{M,L'}}\limsup_nQ^n\psi_n
  \ge\inf_{Q\in\cQ_{M,L}}\limsup_nQ^n\psi_n
.\]
Again 
$\limsup_{n\to\infty} P_*^n \psi_n \ge \liminf_k P_*^{n_k} \psi_{n_k}$.

By \cref{thm:smooth-oscillation}\labelcref{item:sf-density},
every possible stage alternative is a finite mixture of $C^\infty$ densities with compact support inside the slabs,
and hence itself $C^\infty$.
As in the observation following the proof of \cref{thm:diagonal},
this means that the infimum over alternatives
can be restricted to $C^\infty$ densities.

It remains to check the two regularity properties for the null,
which will similarly allow restricting level only over laws with these properties.
Recall that $P_* = \sum_{k \ge 0} w_k R_k$,
with $R_0 = (T_{B_0}) \push U$,
$R_k = (T_{B_k}) \push P^{\eps_k}_{m_k}$,
and $w_k \le 2(1-\theta)4^{-k}$ for $k\ge1$.

For \labelcref{item:reg-density}, take $c = a + \ell$.
Write $f_0$ for the density of $U$,
and $f_k$ for that of $P^{\eps_k}_{m_k}$ for $k \ge 1$.
The density of $R_k$
is zero outside $B_k$,
and $r_k(v) = f_k(T_{B_k}^{-1} v) / \vol(B_k)$ within $B_k$.
By \cref{thm:smooth-oscillation}\labelcref{item:sf-density}, $f_k$ is $C^\infty$ with compact support in $(0,1)^s$, so $r_k$ is $C^\infty$ on $\R^s$ with compact support in $B_k$.
We also have $\norm{r_k}_\infty \le \norm{f_k}_\infty / \vol(B_k) \le 2^{k} H / \bigl((1-\theta)\vol(G)\bigr)$ for $k\ge1$.
Since $w_k \le 2(1-\theta)4^{-k}$ for $k \ge 1$,
$w_k \norm{r_k}_\infty$ is at most $2^{1-k}H/\vol(G)$;
as $w_0 \norm{r_0}_\infty$ is also finite,
the sum $\sum_{k \ge 0} w_k \norm{r_k}_\infty < \infty$.
The series $p_* = \sum_{k \ge 0} w_k r_k$ therefore converges uniformly,
and defines a continuous density.
A point with $z^{(1)}\neq a+\ell$ has a neighbourhood meeting only finitely many $B_k$, on which $p_*$ is a finite sum of $C^\infty$ functions.

For \labelcref{item:reg-kernel},
notice every $B_k$ has the same $X$ and $Y$ sides,
so each $T_{B_k}$ acts on the $(x,y)$ coordinates by the same affine map $\tau$, given by the corresponding coordinates of $T_G$.
For $k\ge1$, it acts on $z$ by an affine bijection whose inverse sends
$z^{(1)}$ to $(z^{(1)}-b_{k-1})/(b_k-b_{k-1})$,
with the endpoints $b_k$ defined in \cref{thm:diagonal}.
By \cref{thm:smooth-oscillation}\labelcref{item:sf-kernel},
a version of the conditional law of $(X,Y)$ given $Z=z$ under $R_k$
is therefore given by $\tau \push \Lambda_{t_k(z)}$, where
\[
  t_k(z) = \eps_k \sin\Bigl( 2\pi m_k \frac{z^{(1)}-b_{k-1}}{b_k-b_{k-1}} \Bigr)
  \quad\text{for }k \ge 1,
  \qquad
  t_0 \equiv 0
.\]
Since the slabs have disjoint $Z$ projections, conditioning on $Z$ identifies the component.
Thus a version of the conditional law under $P_*$ is $\kappa^*(z)=\tau\push\Lambda_{t(z)}$, where $t(z)=t_k(z)$ when $b_{k-1}<z^{(1)}\leq b_k$ for some $k\geq1$, and $t(z)=0$ otherwise.
The function $t$ depends on $z$ only through $z^{(1)}$,
is continuous on each slab,
and vanishes at both ends of each slab
because $\sin(2 \pi m_k \cdot 0) = \sin(2 \pi m_k \cdot 1) = 0$;
$t$ is therefore continuous away from $\{z^{(1)}=a+\ell\}$.
It is continuous at $z^{(1)} = a+\ell$ as well,
since $\abs t \le \eps_k = \eps_0 2^{-k}$ on the $k$th slab,
and it is zero outside $\{b_0<z^{(1)}<a+\ell\}$.
As a function of the single coordinate $z^{(1)}$,
$t$ is continuous and compactly supported,
hence uniformly continuous.
The same holds when it is viewed as a function on $\R^{d_Z}$.
By \eqref{eq:smooth-tv}, and since $\tau\push$ preserves total variation,
$\kappa^*$ is uniformly continuous.
\end{proof}

The stage alternatives also admit conditional densities that are smooth in $(x,y,z)$, with bounded derivatives of every fixed order.
By \cref{thm:smooth-oscillation}\labelcref{item:sf-smooth-kernel}, each component has such a conditional density.
Keep it on the compact support of the component's $Z$ marginal, and interpolate to the common baseline $\tau\push\Lambda_0$ using smooth cutoffs supported inside its slab.
The supports are disjoint and only finitely many components are involved, so these interpolations define a probability kernel everywhere and agree with the required conditional law almost surely.
The conditional densities have a common compact support in $(x,y)$ by \cref{thm:smooth-oscillation}\labelcref{item:sf-smooth-kernel}, so their bounded first derivatives in $z$ give uniformly continuous conditional laws in total variation.

The regular witness is again available for all $L$ such that $\cQ_{M,L}$ is nonempty.

Uniform continuity gives the existence of a modulus of continuity for each law;
we do \emph{not} have a common modulus over the model.
Instead, the frequencies $m_k$ are chosen only after the test and the sample sizes are known,
consistent with positive results under quantitative regularity assumptions.

\section{A black-box proof from finite-sample hardness}
\label{sec:black-box}

We can instead prove \cref{thm:main-intro}
based on the finite-sample theorem of \textcite{shahPeters2020};
this approach, though, does not allow the stronger results of
\cref{thm:main,thm:regular}.

\begin{fact}[Shah--Peters domination]
\label{thm:sp-domination}
For every $n$, $M\in(0,\infty]$, test $\phi_n$, and $Q\in\cQ_M$,
\[
  Q^n \phi_n \le \sup_{P \in \cP_M} P^n \phi_n
.\]
\end{fact}
\begin{proof}
Call the right-hand side $a$; we may assume $a<1$.  For each $\alpha\in(a,1)$, $\phi_n$ has level $\alpha$ over $\cP_M$, so Theorem 2 of \textcite{shahPeters2020} gives $Q^n \phi_n \le \alpha$.  Let $\alpha \searrow a$.
\end{proof}

Shah and Peters couple observations from null and alternative laws to be close in Euclidean distance, and control a measurable rejection region by approximation with finitely many boxes.
The null law supplied by \cref{thm:sp-domination} may depend on the test, sample size, and desired error,
which is weaker than the guarantees of an oscillating family (\cref{thm:tensorization}).

To use \cref{thm:sp-domination} inside a mixture, we again convert a test of the mixture into a test of only the active component.

\begin{lemma}[Localization]
\label{thm:localization}
Let $\mu$ be a probability measure on $\R^s$, $w\in(0,1)$, and $\psi_n$ an $n$-sample test.  There is an $n$-sample test $\phi_n$ such that
\[
  \bigl((1-w) \mu+wS\bigr)^n\psi_n=S^n\phi_n
  \qquad\text{for every probability measure $S$ on $\R^s$.}
\]
\end{lemma}

The test $\phi_n$ independently retains each input observation with probability $w$, replaces the others by draws from $\mu$, and averages $\psi_n$ over this randomization.

\begin{proof}
For $\chi\in\{0,1\}^n$ and $v_{1:n},a_{1:n}\in(\R^s)^n$, let $W_i(\chi,v,a)=v_i$ if $\chi_i=1$ and $W_i(\chi,v,a)=a_i$ otherwise.
Letting $\abs\chi = \sum_i \chi_i$,
define
\[
  \phi_n(v_{1:n})
  =\sum_{\chi\in\{0,1\}^n} w^{\abs\chi}(1-w)^{n-\abs\chi}
   \int\psi_n\bigl(W_1(\chi,v,a),\dots,W_n(\chi,v,a)\bigr)\,
   \ud \mu^n(a)
.\]
Each integrand is bounded and jointly measurable, so $\phi_n$ is a measurable map into $[0,1]$.
If $\chi_i\sim\operatorname{Bernoulli}(w)$, $a_i\sim\mu$, and $v_i\sim S$ are all independent, then the $W_i$ are independent with law $(1-w)\mu+wS$, and $\phi_n(v_{1:n})$ is the conditional expectation of $\psi_n(W_{1:n})$ given $v_{1:n}$. The claim follows by integrating over $v_{1:n}$.
\end{proof}

\begin{proposition}[Black-box local replacement]
\label{thm:black-box-local}
Let $B$ be a box,
$Q \in \cQ_\infty$ with $\supp Q \subseteq B$,
$\mu$ a probability measure on $\R^s$,
and $w \in (0,1)$.
For every $n$, test $\psi_n$, and $\eta>0$,
there is a conditionally independent $R \in \cE_\infty$
with $\supp R \subseteq B$ such that
\[
  \bigl((1-w) \mu + w R\bigr)^n \psi_n
  \ge
  \bigl((1-w) \mu + w Q\bigr)^n \psi_n - \eta
.\]
\end{proposition}

\begin{proof}
Let $\phi_n$ be the test of \cref{thm:localization} corresponding to $\psi_n$, $\mu$, and $w$.
Let $T : (-1,1)^s \to B$ be the coordinatewise affine bijection,
similar to $T_B$ of \cref{sec:ingredients} but mapping from $(-1, 1)$.
Pushing forward through $T$ or $T^{-1}$ preserves absolute continuity and conditional (in)dependence, and carries supports inside $(-1,1)^s$ to supports inside $B$ and back.
In particular, $Q' = T^{-1} \push Q \in \cQ_1$.

Let $\phi'_n(u_{1:n}) = \phi_n(T u_1, \dots, T u_n)$ on $((-1,1)^s)^n$, extended by zero,
so that for every $P'$ concentrated on $(-1,1)^s$,
$(P')^n \phi'_n = (T \push P')^n \phi_n$.
By \cref{thm:sp-domination} with $M=1$,
there is $P \in \cP_1$ with $P^n \phi'_n \ge (Q')^n \phi'_n - \eta$.
Then $R = T \push P$ is conditionally independent, $\supp R\subseteq B$,
and
\begin{align*}
     \bigl( (1-w) \mu + w R \bigr)^n \psi_n
  &  = R^n \phi_n
     = P^n \phi'_n
   \ge (Q')^n\phi'_n - \eta
     = Q^n \phi_n - \eta
.\end{align*}
The result follows by again using the conclusion of \cref{thm:localization}.
\end{proof}

\begin{proof}[Alternate proof of \cref{thm:main-intro}]
Replace \cref{thm:local-replacement} with \cref{thm:black-box-local}
in the recursion of \cref{thm:diagonal}.
Take any null law $R_0 \in \cP_M$ supported inside $B_0$, and define
\[
  \rho = \inf_{Q\in\cQ_M} \limsup_n Q^n \psi_n.
\]
At stage $k$, apply \cref{thm:black-box-local} with $\mu_k / (1 - w_k)$
and any conditionally dependent $Q_k \in \cE_\infty$ with $\supp Q_k\subseteq B_k$;
a rescaled smooth alternative from \cref{thm:smooth-oscillation} suffices.
Keep the weight recursion \eqref{eq:tail-budget};
the separate density argument using \eqref{eq:weight-constraint} is no longer needed.
Since $\mu_k$ gives no mass to $\R^{d_X+d_Y}\times (B_k)_Z$, \cref{thm:gluing} still shows that each $\widetilde Q_k$ (but with $D_k$ replaced by $Q_k$) is an alternative and that $P_*$ is a null, and the tail estimate is unchanged.
\end{proof}

This isolates what is needed for the diagonal argument:
\begin{romanenum}
  \item \emph{local replacement}, which swaps a selected alternative component for a null component while approximately preserving one finite-sample rejection probability;
  \item \emph{dilution}, which keeps a mixture in the alternative whenever its selected component is;
  \item \emph{countable gluing}, which keeps $Z$-separated null components in the null.
\end{romanenum}
For other models,
the same argument can apply,
as long as there is both a finite-sample domination theorem
and these closure properties.
Likewise, weak density of the null controls expectations of bounded continuous functions, but does not by itself give the local replacement inequality for arbitrary measurable tests.

\section{A topological proof by Baire category} \label{sec:baire}

The diagonal argument constructs a null whose limiting upper rejection rate is at least the infimum over the alternatives.
We now give another proof of \cref{thm:main}: in a suitable compact family,
such nulls form a dense $G_\delta$ set.
The key is to make every finite-sample rejection probability continuous,
including those of arbitrary measurable tests.
A common density bound and \cref{thm:tensorization} provide this continuity.

\subsection{A compact family of laws} \label{sec:compact-family}

We use the weak-$*$ topology on bounded subsets of $L^\infty$:
$f_k\to f$ means $\int f_k g\to\int fg$ for every $g\in L^1$.
The closed unit ball is compact by the Banach--Alaoglu theorem,
and metrizable because $L^1$ of a cube is separable
\citep[Theorems~3.15 and~3.16]{rudin91}.
We can therefore check closedness and continuity using sequences.

Write $\cU_Z=[0,1]^{d_Z}$, and let
\[
  \Theta=\bigl\{(r,t)\in L^\infty(\cU_Z)^2:r^2\le t\le1\text{ a.e.}\bigr\},
  \qquad
  \Theta_0=\bigl\{(r,t)\in\Theta:t=r^2\text{ a.e.}\bigr\}.
\]
Elements of $\Theta$ satisfy $0\le t\le1$ and $\abs r\le1$ almost everywhere.
Write $\Theta_1=\Theta\setminus\Theta_0$.
Using the sign functions $a_X$ and $a_Y$ from \cref{thm:oscillation},
fix $\eta \in (0, 1]$ and
define
\begin{equation}
  f_{r,t}^\eta(x,y,z)
  = 1 + \eta r(z) \bigl(a_X(x)+a_Y(y)\bigr) + \eta^2 t(z) a_X(x) a_Y(y)
  ,\qquad (x,y,z)\in\cU
  \label{eq:frt}
.\end{equation}
Let $\mu_{r,t}^\eta$ have density $f_{r,t}^\eta$, and put
\[
  \cW^\eta = \{ \mu_{r,t}^\eta : (r,t) \in \Theta \},
  \qquad
  \cW_0^\eta = \{ \mu_{r,t}^\eta : (r,t) \in \Theta_0 \},
  \qquad
  \cW_1^\eta = \cW^\eta \setminus \cW_0^\eta
.\]
The parameters have a simple interpretation:
$\eta r(z)$ is the conditional mean of each sign,
and $\eta^2 t(z)$ is their conditional product moment.
Thus $\eta^2(t(z)-r(z)^2)$ is their conditional covariance, which we require to be nonnegative.
The oscillating family of \cref{thm:oscillation} corresponds to $(r_m, 1)$,
with limit $(0,1)$.

\begin{lemma} \label{thm:structure}
Fix $\eta \in (0, 1]$.
For every $(r,t)\in\Theta$, $f_{r,t}^\eta$ is a probability density bounded by $(1 + \eta)^2$.
Its $Z$ marginal is uniform, and its conditional marginal densities are
$1 + \eta r(z) a_X(x)$ and $1 + \eta r(z) a_Y(y)$.
Moreover,
the map $(r,t)\mapsto \mu_{r,t}^\eta$ is injective,
and
$X \indep_{\mu_{r,t}^\eta} Y \mid Z$ if and only if $t=r^2$ almost everywhere.
\end{lemma}
\begin{proof}
We have $f_{r,t}^\eta \le 1 + 2 \eta + \eta^2 = (1 + \eta)^2$ almost everywhere,
since $\abs r\le1$, $t\le1$, and $\abs{a_X}=\abs{a_Y}=1$.
When $a_X = a_Y$, since $t\ge r^2$ we have
$
  f_{r,t}^\eta
  = 1 \pm 2 \eta r + \eta^2 t
  \ge (1 - \eta)^2 \ge 0
$;
if instead $a_X = -a_Y$,
we have
$f_{r,t}^\eta = 1 - \eta^2 t \ge 0$.
The total mass is one, from $\int a_X = \int a_Y = 0$;
thus $f_{r,t}^\eta$ is a density.
The marginal formulas follow immediately.

The parameters can be recovered from the density,
showing injectivity:
\[
  r(z) = \frac1\eta \int f_{r,t}^\eta(x,y,z) a_X(x) \ud x \ud y
  ,\quad
  t(z) = \frac{1}{\eta^2}  \int f_{r,t}^\eta(x,y,z) a_X(x) a_Y(y) \ud x \ud y
.\]

Since the marginal density on $Z$ is $1$,
by \eqref{eq:ci-density} conditional independence holds exactly when the product of the conditional marginals equals \eqref{eq:frt} almost everywhere.
Their difference is $\eta^2 (r^2 - t) a_X a_Y$;
as $\eta > 0$ and $\abs{a_X a_Y}=1$, conditional independence holds exactly when $t = r^2$ almost everywhere.
\end{proof}

\begin{lemma} \label{thm:compact}
$\Theta$ is weak-$*$ compact and metrizable, and
$\Gamma:(r,t)\mapsto \mu_{r,t}^\eta$ is a homeomorphism onto $\cW^\eta$,
where $\cW^\eta$ carries the weak-$*$ topology on its densities.
On $\cW^\eta$ this topology agrees with both setwise convergence and the weak topology of probability measures.
For every $n$ and every measurable test $\psi_n$,
the map $\mu \mapsto \mu^n \psi_n$ is continuous on $\cW^\eta$.
\end{lemma}
\begin{proof}
$\Theta$ lies in the product of two weak-$*$ compact metrizable unit balls.
To show it is closed, suppose $(r_k,t_k)\to(r,t)$ weak-$*$ with $(r_k,t_k)\in\Theta$.
For each $c \in \mathbb Q$,
the definition of $\Theta$ gives that
$t_k - 2 c r_k + c^2 \ge (r_k - c)^2 \ge 0$ almost everywhere.
For any nonnegative $h \in L^1(\cU_Z)$,
weak-$*$ convergence gives
\[
  \int (t-2cr+c^2) \, h \ud z
  = \lim_k \int (t_k - 2 c r_k + c^2) \, h \ud z \ge 0
,\]
implying $t \ge 2 c r - c^2$ almost everywhere.
Since the rationals are countable, these inequalities hold
for all $c \in \mathbb Q$
outside a single null set.
Taking their supremum, density of $\mathbb Q$ in $\R$
gives that $t \ge 2 r \cdot r - r^2 = r^2$ almost everywhere.
Similarly, $t_k \le 1$ implies $t \le 1$;
thus $(r, t) \in \Theta$.
Therefore $\Theta$ is weak-$*$ compact and metrizable.

For $F\in L^1(\cU)$, Fubini's theorem and \eqref{eq:frt} give
\begin{equation} \label{eq:fubini-frt}
  \int F f_{r,t}^\eta \ud v
  = \int F \ud v + \eta\int r F_1 \ud z + \eta^2\int t F_2 \ud z,
\end{equation}
where the functions
\begin{align*}
  F_1(z)&=\int F(x,y,z)(a_X(x)+a_Y(y))\ud x\ud y,\\
  F_2(z)&=\int F(x,y,z)a_X(x)a_Y(y)\ud x\ud y
\end{align*}
are integrable since $F \in L^1(\cU)$ and $\abs{a_X} = \abs{a_Y} = 1$.
As the first term on the right-hand side of \eqref{eq:fubini-frt} is constant
and each of the latter two is a continuous function of $(r, t)$,
$\Gamma$ is continuous.
By \cref{thm:structure} it is injective;
compactness makes it a homeomorphism onto its image.

Weak-$*$ convergence of these densities implies setwise convergence,
which in turn implies weak convergence of the laws.
As $\cW^\eta$ is metrizable, we can check continuity via sequences,
and so the identity from the weak-$*$ topology to each of the other two is continuous.
Both setwise and weak topologies are Hausdorff, and a continuous bijection from a compact space onto a Hausdorff space is a homeomorphism,
so all three topologies agree on $\cW^\eta$.
Finally, the common density bound and \cref{thm:tensorization} give
$\mu_k^n \psi_n \to \mu^n \psi_n$ whenever $\mu_k \to \mu$ in $\cW^\eta$.
\end{proof}

\subsection{Density and category of the hypotheses} \label{sec:category}

It is unsurprising that $\cW_1^\eta$ is dense in $\cW^\eta$.
\begin{lemma} \label{thm:w1-dense}
$\Theta_1$ is weak-$*$ dense in $\Theta$; hence $\cW_1^\eta$ is dense in $\cW^\eta$.
\end{lemma}
\begin{proof}
Let $(r, t) \in \Theta$ and $\eps \in (0, 1)$;
define $r_\eps = (1-\eps) r$, $t_\eps = (1-\eps) t + \eps \le 1$.
We have
\[
  t_\eps - r_\eps^2
  = (1 - \eps) (t - r^2) + \eps (1 - \eps) r^2 + \eps
  \ge \eps
.\]
Thus $\mu_{r_\eps, t_\eps}^\eta \in \cW_1^\eta$.
We also have $\norm{r_\eps - r}_\infty, \norm{t_\eps - t}_\infty \le \eps$,
so $\mu_{r_\eps, t_\eps}^\eta \to \mu_{r,t}^\eta$ as $\eps \searrow 0$.
\end{proof}

As we have seen before,
conditional dependence
can be approximated by rapidly oscillating conditionally independent laws,
and so $\cW_0^\eta$ is also dense.
\begin{lemma} \label{thm:mixtures}
$\Theta_0$ is weak-$*$ dense in $\Theta$; hence $\cW_0^\eta$ is dense in $\cW^\eta$.
\end{lemma}
\begin{proof}
Fix $(r,t)\in\Theta$.
Let $\cD_j$ be the partition of $\cU_Z$ into dyadic cubes of side $2^{-j}$.
Under uniform Lebesgue measure on $\cU_Z$,
define $r_j = \E[ r \mid \sigma(\cD_j) ]$ and $t_j = \E[ t \mid \sigma(\cD_j) ]$.
These functions give the averages over the corresponding cells;
the sequences $(r_j)_{j\ge1}$ and $(t_j)_{j\ge1}$ are bounded martingales with respect to the increasing $\sigma$-algebras $\sigma(\cD_j)$.
Jensen's inequality and the almost-everywhere pointwise inequalities from $\Theta$ give $r_j^2 \le t_j \le 1$.
Since the dyadic partitions generate the Borel $\sigma$-algebra, martingale convergence implies $(r_j, t_j) \to (r,t)$ in $L^1$;
the uniform bounds $\abs{r_j},\abs r,t_j,t\le1$ make this weak-$*$ convergence as well,
by truncating each $L^1$ test function.
To show weak-$*$ density of $\Theta_0$,
it will thus suffice to show that every $(r_j,t_j)$ is in $\overline{\Theta_0}$.

Fix any $j$.
For each $m \ge 1$,
we will define a function $r_{j,m}$ as follows.
For any input $z$,
let $C$ be the cell of $\cD_j$ in which it lies;
let $r_j = \rho_C$, $t_j = \tau_C^2$ on this cell.
If $\tau_C=0$, then $\rho_C=0$; define $r_{j,m}=0$ on $C$.
Otherwise, put
\[
  w_C = \frac12 \left(1 + \frac{\rho_C}{\tau_C} \right) \in [0,1]
  ,\quad
  r_{j,m}(z) = \tau_C \left( 2 \indic_{[0,w_C)}(\{mz^{(1)}\}) - 1 \right)
,\]
where $\{u\}=u-\floor u$.
Then $r_{j,m}^2=t_j$, so $(r_{j,m},t_j)\in\Theta_0$.

On $C$, we have $r_{j,m} - \rho_C = 2 \tau_C \, g_m(z^{(1)})$
for the centred periodic function
$g_m = \indic_{[0,w_C)}(\{m\,\cdot\}) - w_C$,
since $2 \tau_C w_C = \tau_C + \rho_C$.
We have $\abs{g_m} \le 1$,
and $g_m$ has mean zero over each period of length $1/m$;
thus only the two partial periods at the ends of an interval contribute,
and its integral over any interval has absolute value at most $2/m$.
For every $G\in L^1(\cU_Z)$ the function $v\mapsto G(z)\indic_C(z)$ lies in $L^1(\cU)$,
so \cref{thm:oscillation-1d} gives $r_{j,m}\to\rho_C$ weak-$*$ on $C$.
Summing over the finitely many cells gives $r_{j,m}\to r_j$ weak-$*$.
As each $(r_{j,m}, t_j) \in \Theta_0$,
this means $(r_j,t_j)\in\overline{\Theta_0}$.
\end{proof}

Even more,
the weak-$*$ topology on $\cW^\eta$
makes $\cW_0^\eta$ more prominent than $\cW_1^\eta$.
Recall that a set is \emph{nowhere dense} if its closure has empty interior,
\emph{meagre} if it is a countable union of nowhere dense sets,
and \emph{comeagre} if its complement is meagre.
A $G_\delta$ set is a countable intersection of open sets,
and an $F_\sigma$ set is a countable union of closed sets.
In a nonempty compact metrizable space,
the Baire category theorem shows countable intersections of dense open sets are dense,
and hence comeagre sets are dense
\citep[Section 8]{Kechris1995}.

\begin{proposition} \label{thm:category}
$\cW_0^\eta$ is $G_\delta$ and comeagre in $\cW^\eta$; it is not $F_\sigma$.
$\cW_1^\eta$ is $F_\sigma$ and meagre.
\end{proposition}
\begin{proof}
By the injectivity in \cref{thm:structure}, the nonnegative functional
$
  \Phi(\mu_{r,t}^\eta) = \int_{\cU_Z}(t - r^2) \ud z
$
is well-defined on $\cW^\eta$, and vanishes exactly on $\cW_0^\eta$.
The map $t\mapsto \int t$ is weak-$*$ continuous.
Consider the collection of continuous functions
$r \mapsto \int (2 r g - g^2) \ud z$,
and notice that
since $2 r g - g^2 = r^2 - (r - g)^2$,
their supremum over $g \in L^\infty(\cU_Z)$
is $\int r^2$;
thus $r \mapsto \int r^2$ is lower semicontinuous.
Hence $\Phi$ is upper semicontinuous,
and so the sets $A_j := \{ \Phi \ge 1 / j \}$ are closed for each $j \ge 1$.

Each $A_j$ is nowhere dense:
it misses $\cW_0^\eta = \{ \Phi = 0 \}$, which is dense by \cref{thm:mixtures}.
Thus $\cW_1^\eta = \bigcup_{j \ge 1} \{ \Phi \ge 1 / j \}$
is $F_\sigma$ and meagre;
its complement $\cW_0^\eta$ is $G_\delta$ and comeagre.

Suppose $\cW_0^\eta$ is $F_\sigma$:
it can be written as a countable union of closed sets.
Each of these closed sets must miss the set $\cW_1^\eta$,
which is dense by \cref{thm:w1-dense},
making $\cW_0^\eta$ meagre.
Thus $\cW^\eta = \cW_0^\eta \cup \cW_1^\eta$ is also meagre,
and so its complement, the empty set, is dense by the Baire category theorem.
Since $\cW^\eta$ is nonempty, this cannot be true;
thus $\cW_0^\eta$ is not $F_\sigma$.
\end{proof}

Here $\eta^2\Phi$ is the average conditional covariance of the two signs.
Along the oscillating family $\mu_{r_m,1}^\eta$, $r_m\to0$ weak-$*$ but $r_m^2=1$,
so $\Phi(\mu_{r_m,1}^\eta)=0$ while $\Phi(\mu_{0,1}^\eta)=1$.
This is how a limit of null laws becomes dependent: weak-$*$ convergence preserves the conditional means in the integrated sense, but need not preserve their squares.

\subsection{Generic witnesses} \label{sec:generic}

\begin{theorem} \label{thm:generic}
For any test sequence $(\psi_n)$, let
$\rho = \inf_{Q\in\cW_1^\eta} \limsup_n Q^n \psi_n$.
Then
$
  \bigl\{ P \in \cW_0^\eta : \limsup_{n\to\infty} P^n \psi_n \ge \rho \bigr\}
$
is a dense $G_\delta$ subset of $\cW^\eta$.
In particular, if the test is consistent at every alternative in $\cW_1^\eta$,
then $\limsup_n P^n \psi_n = 1$ for a comeagre set of nulls in $\cW^\eta$,
and (a fortiori) there is at least one $P_*$ for which $\limsup_n P_*^n \psi_n = 1$.
\end{theorem}
\begin{proof}
For $j,m\ge1$, let
$
  O_{j,m} = \bigcup_{n\ge m} \{ \mu \in \cW^\eta : \mu^n \psi_n > \rho - \tfrac1j \}
$.
Each of these sets is open by \cref{thm:compact}.
By the definition of $\rho$,
each of these sets contains $\cW_1^\eta$,
and hence is dense by \cref{thm:w1-dense}.
Thus
$
  O =
  \{ \mu \in \cW^\eta : \limsup_n \mu^n \psi_n \ge \rho\}
  = \bigcap_{j, m \ge 1} O_{j,m}
$
is a dense $G_\delta$ set by the Baire category theorem.
$\cW_0^\eta$ is dense (\cref{thm:mixtures}) and $G_\delta$ (\cref{thm:category}),
so
$O \cap \cW_0^\eta = \{ P \in \cW_0^\eta : \limsup_n P^n \psi_n \ge \rho \}$
is also dense and $G_\delta$ by another application of the Baire category theorem.
Its complement is a countable union of nowhere-dense sets,
hence $O \cap \cW_0^\eta$ is comeagre.
As $\cW^\eta$ is a nonempty compact metrizable space, it is not meagre in itself;
this implies $O \cap \cW_0^\eta$ has at least one element.
\end{proof}

\begin{proof}[Alternate proof of \cref{thm:main}]
Let $L > (2M)^{-s}$ be given,
and choose $\eta$ and a box $G$ with $\overline G\subset(-M,M)^s$
such that $(1+\eta)^2/\vol(G) \le L$;
for $M = \infty$ it suffices to take $G$ large.
Using the affine map $T_G$ from \cref{sec:ingredients},
put $\cW^\eta_G = \{ (T_G) \push \mu : \mu \in \cW^\eta \}$.
$T_G$ preserves conditional (in)dependence, and divides density bounds by $\vol(G)$.
Thus $\cW^\eta_G\subseteq\cE_{M,L}$,
$\cW^\eta_G\cap\cP_M=(T_G)\push\cW^\eta_0\subseteq\cP_{M,L}$,
and $\cW^\eta_G\cap\cQ_M=(T_G)\push\cW^\eta_1\subseteq\cQ_{M,L}$.
Since $\bigl((T_G)\push \mu\bigr)^n\psi_n = \mu^n\bigl(\psi_n\circ T_G^{\times n}\bigr)$,
applying \cref{thm:generic} to the pulled-back tests $\psi_n\circ T_G^{\times n}$ gives
$P_*\in\cW^\eta_G\cap\cP_M$ with
\[
  \limsup_nP_*^n\psi_n
  \ge\inf_{Q\in\cW^\eta_G\cap\cQ_M}\limsup_nQ^n\psi_n
  \ge\inf_{Q\in\cQ_{M,L}}\limsup_nQ^n\psi_n.
\]
This is the inequality of \cref{thm:main}; its remaining claims follow as in \cref{sec:main-proof}.
\end{proof}

\Cref{thm:category,thm:generic}
do not assert generic failure in the full density model,
but only in the relative topology on this particular family.
The affine pushforward is a continuous injection from the compact space $\cW^\eta$ into the Hausdorff space of laws on $\R^s$ with the weak topology,
so it is a homeomorphism onto $\cW^\eta_G$ and carries the category conclusions to that family.
The result does strengthen the existence conclusion:
every relative neighbourhood in $\cW^\eta_G$ contains a null witness,
and all laws in this family have the same uniform $Z$ marginal.
The argument does not give the regularity of \cref{thm:regular};
for that, the explicit construction controls both the amplitudes of the oscillations and the masses of their slabs.

The category result also identifies the obstruction to the criterion of \textcite{boekenEtAl2026}:
consistent finite-precision testability requires both hypotheses to be $F_\sigma$ in the relative weak topology on their union.
By \cref{thm:category}, the null in $\cW^\eta_G$ is not $F_\sigma$, whereas the alternative is.
If $\cP_{M,L}$ were $F_\sigma$ in $\cE_{M,L}$ for some $L>(2M)^{-s}$,
its intersection with $\cW^\eta_G\subseteq\cE_{M,L}$ would be $F_\sigma$ there, a contradiction.
Thus the full null cannot be exhausted by countably many subsets that are weakly closed in $\cE_{M,L}$.

This answers the density-model question of \textcite{boekenEtAl2026} in the negative.
Their tests have open acceptance and rejection regions, with suspension of judgement allowed elsewhere.
Merging suspension with acceptance would give a measurable test whose rejection probability tends to zero at every null and to one at every alternative,
also ruled out directly by \cref{thm:main-intro}.

The compact family explains why our topological argument applies even to arbitrary measurable tests.
On $\cW^\eta_G$, weak convergence agrees with weak-$*$ convergence of the bounded densities,
and \cref{thm:tensorization} makes every finite-sample rejection probability continuous.
Weak convergence on the unrestricted space of Borel laws does not provide this continuity for measurable tests.
Related $F_\sigma$ criteria for measurable tests appear in \textcite{demboPeres1994},
\textcite[Corollary~9.4.23]{kleijn:freqbayes}, and \textcite{ermakov:consistent}.

\begin{appendix}
\crefalias{section}{appendix}
\crefalias{subsection}{appendix}

\section{Discrete conditioning variables}
\label{sec:discrete-z}

We now prove \cref{thm:consistent-ind-test,thm:discrete-consistent}.

\begin{proposition}[store*=thm:consistent-ind-test]
  Let $\cX, \cY$ be Polish spaces,
  $\cM$ the set of Borel probability measures on $\cX \times \cY$,
  and $\cP \subseteq \cM$ the distributions where $X \indep Y$.
  For any $\alpha \in (0, 1)$,
  there exists a strongly consistent sequence of tests
  distinguishing $\cP$ from $\cM \setminus \cP$
  with finite-sample level $\alpha$.
\end{proposition}
\begin{proof}
  Let $T_\cX$ be a continuous, injective mapping from $\cX$ to $\ell_2$;
  such a mapping is guaranteed to exist by the Urysohn embedding theorem \parencite[Theorem 4.14]{Kechris1995}
  and any continuous embedding of $[0, 1]^{\mathbb N}$ in $\ell_2$,
  e.g.\ $(z_i)_{i\ge1} \mapsto (2^{-i} z_i)_{i\ge1}$.
  By Proposition 5.2 of \textcite{Ziegel2024May},
  the (bounded, continuous) kernel
  $k_\cX(x, x') = \exp\left( - \norm{T_\cX(x) - T_\cX(x')}_{\ell_2}^2 \right)$
  is integrally strictly positive definite (ispd).
  Let $\varphi_\cX$ be a feature map, $k_\cX(x, x') = \inner{\varphi_\cX(x)}{\varphi_\cX(x')}$.
  Define $k_\cY$, $\varphi_\cY$ analogously,
  and let $\varphi_{\cX \cY} = \varphi_\cX \otimes \varphi_\cY$.
  The Hilbert--Schmidt Independence Criterion \citep{hsic} is
  \[
    \Delta(P)
    = \norm*{ P \varphi_{\cX \cY} - (P_X \, \varphi_\cX) \otimes (P_Y \, \varphi_\cY)}^2
  ;\]
  the expectations exist since the kernels are bounded,
  and
  $\Delta(P) > 0$ iff $X \nindep Y$ since the kernels are ispd \citep{Szabo2018}.
  Letting $\hat P = \frac1n \sum_{i=1}^n \delta_{x_i} \otimes \delta_{y_i}$,
  the strong law of large numbers in separable Hilbert spaces
  gives $\Delta(\hat P) \to \Delta(P)$ almost surely.

  For $n=1$, let the test accept; thus assume $n\ge2$.
  Let $\pi$ be a permutation
  and $\hat P^{(\pi)} = \frac1n \sum_{i=1}^n \delta_{x_i} \otimes \delta_{y_{\pi_i}}$.
  Taking $\Pi$ to be uniform over the $n!$ permutations,
  let $\hat q_{1-\alpha}$ be the $(1-\alpha)$th quantile of
  $\Delta(\hat P^{(\Pi)})$.
  By a standard argument,
  if $X \indep Y$
  then $\Delta(\hat P) > \hat q_{1-\alpha}$ with probability at most $\alpha$;
  \citet{Hemerik2018} give a useful generalized account.
  For strong consistency, though,
  we need asymptotic size $0$, not $\alpha$.
  We will thus use $\alpha_n = \alpha / n^2$ for $\psi_n$;
  each test has level $\alpha_n \le \alpha$,
  preserving the finite-sample level.
  As $\sum_{n \ge 1} \alpha_n = \pi^2 \alpha / 6$ is finite,
  the Borel--Cantelli lemma shows the test falsely rejects
  an almost surely finite number of times.

  For the alternative, when $\Delta(P) > 0$,
  we must understand the quantile $\hat q_{1 - \alpha_n}$.
  We will do so with a concentration inequality
  on $\Delta(\hat P^{(\Pi)}) \mid \hat P$,
  and begin by showing that, regardless of $\hat P$,
  $\E[ \Delta(\hat P^{(\Pi)}) \mid \hat P ] < 1 / n$.\footnote{%
    If $n\alpha_n\to\infty$, this result and Markov's inequality give
    $\hat q_{1-\alpha_n}\le1/(n\alpha_n)\to0$.
    Our summable choice of $\alpha_n$ requires a sharper bound.
  }
  To evaluate this,
  notice that \citep{hsic}
  \[
    \Delta(\hat P^{(\pi)})
    = \frac{1}{n^2} \inner{K_\cX}{H K_\cY^{(\pi)} H}_F
    \quad
    \text{for }
    [K_\cX]_{ij} = k_\cX(x_i, x_j)
    \text{, }
    [K_\cY^{(\pi)}]_{ij} = k_\cY(y_{\pi_i}, y_{\pi_j})
  ,\]
  where $H = I - \frac1n \bone\bone\tp$ is the centring matrix
  and $\inner{A}{B}_F = \sum_{ij} A_{ij} B_{ij}$ the Frobenius inner product.
  By linearity,
  \[
    \E_\Pi[ \Delta(\hat P^{(\Pi)}) \mid \hat P]
    = \frac{1}{n^2} \inner{K_\cX}{H \E[ K_\cY^{(\Pi)} \mid \hat P] H}_F
  .\]
  Call the inner expectation $M$;
  it has diagonal $1$
  and constant off-diagonal
  \[
    \ell_Y = \frac{1}{n(n-1)} \sum_{i \ne j}^n k_\cY(y_i, y_j),
  \]
  so
  $M = \ell_Y \bone\bone\tp + (1 - \ell_Y) I$.
  Then $
    H M H
    = (1-\ell_Y) H
  $,
  and so defining $\ell_X$ analogously,
  \[
    \E_\Pi[ \Delta(\hat P^{(\Pi)}) \mid \hat P]
    = \frac{1 - \ell_Y}{n^2} \inner{K_\cX}{H}_F
    = \frac{1 - \ell_Y}{n^2} \left( (n -1) (1 - \ell_X) \right)
    < \frac1n
  \]
  since $\ell_X, \ell_Y \in (0, 1]$.

  Finally, we show the quantile is not too far from the mean.
  Let $f(\pi) = \sqrt{\Delta(\hat P^{(\pi)})} = \norm{\hat P^{(\pi)} \varphi_{\cX\cY} - (\hat P_X \varphi_\cX) \otimes (\hat P_Y \varphi_\cY)}$;
  we already know that $\E f(\Pi) \le \sqrt{\E \Delta(\hat P^{(\Pi)})} \le 1 / \sqrt{n}$.
  Since every $\varphi_{\cX\cY}(x, y)$ has unit norm,
  $\abs{f(\pi) - f(\pi')} \le \norm{(\hat P^{(\pi)} - \hat P^{(\pi')}) \varphi_{\cX\cY}} \le 2 \, d(\pi, \pi')$
  for the normalized Hamming distance $d(\pi, \pi') = \frac1n \abs{\{ i : \pi_i \ne \pi'_i \}}$;
  in particular $f$ changes by at most $4/n$ when two coordinates of $\pi$ are transposed.
  Concentration on the symmetric group
  \citep[Theorem 5.2.6]{Vershynin2026Jan}
  then implies that
  $\Pr( \abs{f(\Pi) - \E f(\Pi)} \ge t) \le 2 \exp(- c n t^2)$
  for a constant $c$,
  and thus $\hat q_{1-\alpha_n} \le \left( \frac{1}{\sqrt n} + \sqrt{\frac{\log(2/\alpha_n)}{c n}} \right)^2 = \mathcal O\left( \frac{\log n}{n} \right)$.
  We therefore know $\hat q_{1-\alpha_n} \to 0$
  as $n \to \infty$ deterministically,
  while almost surely $\Delta(\hat P) \to \Delta(P) > 0$.
  Thus, almost surely $\Delta(\hat P)$ eventually exceeds $\hat q_{1-\alpha_n}$,
  showing strong consistency.
\end{proof}

\begin{proposition}[store*=thm:discrete-consistent]
  Let $\cX, \cY, \cZ$ be Polish spaces,
  $\cD$ the set of all Borel probability measures on $\cX \times \cY \times \cZ$
  whose marginal on $\cZ$ is discrete,
  and $\cP \subseteq \cD$ the distributions with $X \indep Y \mid Z$.
  For any $\alpha \in (0, 1)$,
  there exists a strongly consistent sequence of tests
  distinguishing $\cP$ from $\cD \setminus \cP$
  with finite-sample level $\alpha$.
\end{proposition}
\begin{proof}
  We will use as a subroutine the independence test of \cref{thm:consistent-ind-test};
  let $\phi^{(a)}$ denote that test sequence at finite-sample level $a$,
  so $\phi_m^{(a)}$ has level $a/m^2$.

  Denote the $k \le n$ unique values of $z$ observed by $\psi_n$
  by $z_{(1)}, \dots, z_{(k)}$
  in order of first appearance:
  $z_{(1)} = z_1$,
  $z_{(j+1)} = z_{\min\{ i \ge 1 : z_i \notin \{ z_{(1)}, \dots, z_{(j)} \} \}}$.
  Partition the $n$ inputs into the $k$ subsequences
  $g_{i,n} = ((x_j,y_j):1\le j\le n,\ z_j=z_{(i)})$,
  retaining observation order and any repeated values.
  Then let $\psi_n = \min\{ 1, \sum_{i=1}^k \phi_{\abs{g_{i,n}}}^{(\alpha 2^{-i})}(g_{i,n}) \}$.

  Let $P \in \cP$.
  Conditional on the full sequence $Z_1,Z_2,\dots$,
  each group $g_{i,n}$ contains an increasing number of samples from $P_{XY \mid Z = z_{(i)}}$, with each group independent.
  By $\phi$'s finite-sample level,
  $\E\bigl[ \phi_{\abs{g_{i,n}}}^{(\alpha 2^{-i})}(g_{i,n}) \mid Z_1, Z_2, \dots \bigr] \le \alpha 2^{-i}$;
  therefore
  $P^n \psi_n \le \sum_{i \ge 1} \alpha 2^{-i} = \alpha$,
  showing the finite-sample level.
  For strong consistency,
  instead argue via the Borel--Cantelli lemma
  that almost surely
  only a finite number of (group $i$, size $m$) pairs
  ever reject:
  $\sum_{i,m} \alpha / (2^i m^2) = \pi^2 \alpha / 6 < \infty$.
  Each such pair causes $\psi_n$ to reject only while $\abs{g_{i,n}} = m$,
  which corresponds to only finitely many $n$ values,
  since each observed value is an atom of $P_Z$
  and so $\abs{g_{i,n}} \to \infty$.
  Thus almost surely, $\psi_n$ falsely rejects only finitely often.

  For $Q \in \cD \setminus \cP$,
  there is an atom $z$ with $Q_Z(\{z\}) > 0$
  such that the conditional distribution $Q_z := Q_{X Y \mid Z = z}$ is dependent.
  Let $J$ be the random rank of $z$'s first appearance,
  which is almost surely finite.
  Conditional on the full sequence $Z_1,Z_2,\dots$,
  group $J$ is an i.i.d.\ sequence from $Q_z$
  tested at the fixed level $\alpha 2^{-J}$.
  Again $\abs{g_{J,n}} \to \infty$ almost surely,
  and so strong consistency of $\phi$ implies $\psi_n \to 1$ almost surely.
\end{proof}

\section{Lossless feature selection}
\label{sec:gyorfi-walk}

\Textcite[Section~6]{gyorfiEtAl2023} ask whether a universally strongly consistent test of losslessness exists.
Losslessness of a transformation $T$ for predicting $Y$ from $X$ is equivalent to
$X\indep Y\mid T(X)$ \parencite[Section~2.2]{gyorfiEtAl2023}.

For a Euclidean-valued transformation $Z=T(X)$, the augmented law of $(X,Y,Z)$ is singular,
so our density-model theorem cannot be applied to that triple directly.
For coordinate selection in particular, however, we can easily work around this issue.
For $S\subseteq\{1,\dots,d_X\}$, write $X_S$ for the subvector of $X$ with coordinates in $S$.
If $S$ and $S^c$ are both nonempty, then
\begin{equation}
  X\indep Y\mid X_S
  \quad\text{iff}\quad
  X_{S^c}\indep Y\mid X_S.
  \label{eq:feature-selection}
\end{equation}
The triple $(X_{S^c},Y,X_S)$ is a coordinate permutation of $(X,Y)$,
and any test based on $(X,Y,X_S)$ is a measurable test of this triple.
Thus, by \cref{thm:main-intro} applied to this triple, with dimensions $\abs{S^c}$, $d_Y$, $\abs{S}$ and any finite $M$,
no test of losslessness for coordinate selection can have pointwise asymptotic level $\alpha<1$ and consistency against every alternative over all absolutely continuous bounded-support laws of $(X,Y)$.
In particular there is no strongly consistent test on this model,
which resolves the unrestricted existence question of \textcite[Section~6]{gyorfiEtAl2023} for coordinate selection.

Our result does not imply impossibility for every transformation: if $T$ is injective with a measurable inverse on its image, then $X$ is a function of $Z$ and losslessness holds for every law.

\end{appendix}

\begin{acks}[Acknowledgments]
  The author is also affiliated with the Alberta Machine Intelligence Institute.
  She would like to thank Ilmun Kim for drawing her attention to this problem,
  and Zheng He, Feng Liu, Aaron Wei, and Nathaniel Xu for inspiring conversations.
  Claude (Fable 5.1 and Opus 5) and GPT (6 Astra and 5.6 Sol)
  were used in the preparation of this manuscript;
  the author has carefully verified and edited all large language model outputs,
  and takes full responsibility for the paper.
\end{acks}
\begin{funding}
  This work was supported in part by the Canada CIFAR AI Chairs program.
\end{funding}

\printbibliography[heading=bibintoc]

\end{document}